\documentclass{amsart}

\usepackage[margin=30truemm]{geometry}
\usepackage[utf8]{inputenc}

\usepackage{amsthm}
\usepackage{amsmath}
\usepackage{amsfonts}
\usepackage{amssymb}
\usepackage{amscd}
\usepackage{mathrsfs}
\usepackage{txfonts}
\usepackage{tikz}
\usepackage{ascmac}
\usepackage{multirow}
\usepackage{graphicx}

\usepackage{mathtools}
\usepackage{cleveref}
\usepackage{url}
\usepackage{tikz-cd}
\usepackage{enumitem}
\usepackage{subcaption}
\usepackage{afterpage}
\usepackage{changepage}
\usepackage{longtable}
\usepackage{bm}
\usepackage{multicol}
\usepackage{wrapfig}

\usepackage{lineno}

\theoremstyle{plain}
    
    \newtheorem{theorem}{Theorem}[section]
    \newtheorem{proposition}[theorem]{Proposition}
    \newtheorem{corollary}[theorem]{Corollary}
    \newtheorem{lemma}[theorem]{Lemma}
    \newtheorem{claim}[theorem]{Claim}
    
    \newtheorem{conjecture}[theorem]{Conjecture}

\theoremstyle{definition}
    \newtheorem{definition}[theorem]{Definition}

    \newtheorem{example}[theorem]{Example}
    \newtheorem{question}[theorem]{Question}

\theoremstyle{remark}
	\newtheorem{remark}[theorem]{Remark}%

\crefname{maintheorem}{Theorem}{Theorems}
\crefname{claim}{Claim}{Claims}

\newcommand{\R}{\mathbb{R}}
\newcommand{\Z}{\mathbb{Z}}

\newcommand{\id}{\mathop{\mathrm{id}}\nolimits}

\newcommand{\Int}{\mathrm{Int}\,}

\title{
    On the classification of 2-plat 2-knots
}

\author{
    Jumpei Yasuda
}

\date{
    \today
}

\address{Department of Mathematics, Graduate School of Science, Osaka Metropolitan University, Osaka 558‐8585, JAPAN}
\email{j.yasuda@omu.ac.jp}


\begin{document}
\maketitle

\begin{abstract}
An $n$-plat 1-knot is one isotopic to the plat closure of some $2n$-braid, which is also called an $n$-bridge 1-knot.
Schubert classified 2-bridge 1-knots by considering their double branched covers which are homeomorphic to lens spaces.
A 2-knot is a 2-sphere smoothly embedded in 4-space or 4-sphere.
An $n$-plat 2-knot is one isotopic to the plat closure of some 2-dimensional $2n$-braid.

The aim of this paper is to classify 2-plat 2-knots.
By a result of Montesinos, double branched covers do not distinguish 2-plat 2-knots.
Thus, we introduce a new invariant to classify them.
Our invariant serves as an analogue of a torsion invariant.
Furthermore, it is an obstruction to invertibility of 2-knots.
\end{abstract}

\section{Introduction}\label{Section: Introduction}
The \textit{plat closure} of a $2n$-braid $\beta$, denoted by $\tilde{\beta}$, is the 1-knot or 1-link in $\R^3$ obtained by attaching $2n$ arcs to its endpoints, as shown in \Cref{fig:Dig-PlatClosure}.
A 1-knot is called \textit{$n$-plat} if it is ambiently isotopic to the plat closure of some $2n$-braid.
Every 1-plat 1-knot is trivial.
2-plat 1-knots are also known as 2-bridge 1-knots, and they were first studied by Bankwitz and Schumann \cite{Bankwitz-Schumann1934} as \textit{Viergeflechte} (\textit{4-plats}).

In \cite{Schubert1954}, Schubert introduced a normal form $S(p, q)$ of 2-bridge 1-knots for coprime integers $p,q \in \Z$, where $p > 0$ is odd.
Moreover, he gave a complete classification of 2-bridge 1-knots as follows:
$S(p, q)$ and $S(p', q')$ are ambiently isotopic if and only if the following hold: (1) $p = p'$ and (2) $q' \equiv q^{\pm 1} \mod p$.
This corresponds to the classification of lens spaces $L(p, q)$ arising as double covers of $S^3 = \R^3 \cup \{\infty\}$ branched along $S(p, q)$.

\begin{figure}[h]
    \centering
    \includegraphics[width = 0.5\hsize]{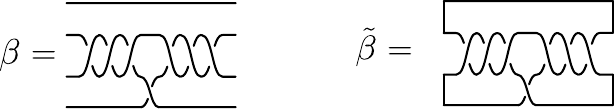}
    \caption{The plat closure $\tilde{\beta}$ of a $4$-braid $\beta$.}
    \label{fig:Dig-PlatClosure}
\end{figure}

A \textit{2-knot} is a 2-sphere $S^2$ smoothly embedded in $\R^4$.
Two 2-knots are said to be \textit{equivalent} if they are smoothly isotopic in $\R^4$.
We work in either the smooth category or the piecewise linear (PL) category.
Surfaces embedded in $\R^4$ are assumed to be locally flat in the PL category.

In \cite{Yasuda21}, the author introduced the plat closures of 2-dimensional braids as a higher-dimensional analogue of the plat closures of classical braids (\Cref{Subsection:Plat closure of 2-dim braid}).
Every 2-knot is equivalent to the plat closure of some 2-dimensional braid.
A 2-knot $F$ is called \textit{$n$-plat} if it is equivalent to the plat closure of some 2-dimensional $2n$-braid.
Then, every 1-plat 2-knot is trivial (\cite[Theorem 5.5]{Yasuda21}), and every 2-plat 2-knot is ribbon (\Cref{2-plat 2-knot is ribbon}, see also \cite[Theorem 5.7]{Yasuda21}).

The aim of this paper is to classify oriented 2-plat 2-knots.
To this end, we introduce normal forms of oriented 2-plat 2-knots and define a new invariant $\tau(F)$ of oriented 2-knots.

\subsection{Main results}
Let $p$ and $q$ be coprime integers, where $p > 0$ is odd.
Take a continued fraction expansion $[c_1, c_2, \dots, c_m]$ of $q/p$.
An oriented ribbon 2-knot $F(p, q)$ is defined by a ribbon presentation as shown in \Cref{fig:Intro-1} with a canonical orientation (\Cref{Subsection: normal forms of 2-plat 2-knots}), where each box in the diagram has horizontal $|c_i|$ positive crossings if $c_i \geq 0$ and $|c_i|$ negative crossings if $c_i < 0$, respectively.
The equivalence class of $F(p,q)$ is independent of the choice of a continued fraction expansion (\Cref{Well-definedness of normal form for 2-plat 2-knots}).
Then, the family $F(p, q)$ gives a normal form for 2-plat 2-knots:

\begin{theorem}\label{Main: normal form for 2-plat 2-knot}
    Every oriented 2-plat 2-knot is equivalent to $F(p, q)$ for some coprime $p,q \in \Z$ with $p>0$ odd.
    Conversely, for any coprime $p, q \in \Z$ with $p > 0$ odd, $F(p,q)$ is an oriented 2-plat 2-knot.
\end{theorem}

\begin{figure}[h]
    \begin{minipage}{0.45\hsize}
        \centering
        \includegraphics[width = \hsize]{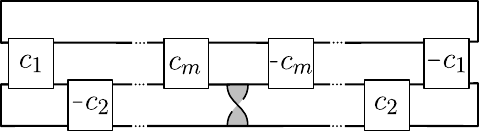}
        \caption{A ribbon presentation of a 2-plat 2-knot $F(p, q)$.}
        \label{fig:Intro-1}
    \end{minipage}
    \quad
    \begin{minipage}{0.45\hsize}
        \centering
        \includegraphics[width = \hsize]{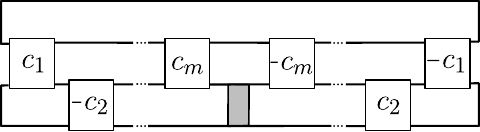}
        \caption{A ribbon presentation of a spun 2-knot $\sigma K(p, q)$.}
        \label{fig:Intro-2}
    \end{minipage}
\end{figure}

We remark that $F(p,q)$ consists of a 2-sphere and a Klein bottle if $p$ is even, and $F(1,q)$ is a trivial 2-knot.

We now turn to the classification of oriented 2-plat 2-knots.
We prove in \Cref{2-plat 2-knots modulo p} that $F(p, q)$ and $F(p,q')$ are equivalent if $q \equiv q' \pmod p$.
Hence, we may assume that $q$ is even.
Then, we obtain a formula for the Alexander polynomial $\Delta_{F}(t)$ of oriented 2-plat 2-knots:

\begin{proposition}\label{Main: Formula of Alexander polynomials for 2-plat 2-knots}
    Suppose that $p$ is odd and $q$ is even.
    Then, we have
    \begin{align*}
        \Delta_{F(p,q)}(t) ~\doteq~ 1 + \sum_{i=1}^{p-1} (-1)^i t^{\, d(i)}, \quad d(i) = \sum_{k=1}^{i} (-1)^{\lfloor kq/p \rfloor}.
    \end{align*}
    Here, $\lfloor x \rfloor$ is the integer part of $x$, and $f(t) \doteq g(t)$ means that $f(t) = \pm t^{n} g(t)$ for some $n \in \Z$.
\end{proposition}

\begin{corollary}\label{Determinant of 2-plat 2-knot}
    The determinant of $F(p,q)$ is equal to $p$.
    In particular, $F(p,q)$ and $F(p',q')$ are equivalent only if $p = p'$.
\end{corollary}

\begin{proof}
    It is known that the determinant of a 2-knot $F$ is equal to $|\Delta_F(-1)|$.
    In addition, the equality $d(i) \equiv i \pmod 2$ holds for each $i$.
    Therefore, we have
    \begin{align*}
        \left|\Delta_F(-1)\right| ~=~ \left|1 + \sum_{i=1}^{p-1} (-1)^i(-1)^{d(i)}\right| ~=~ p.
    \end{align*}
\end{proof}

Since 2-plat 2-knots are ribbon 2-knots of 1-fusion (\Cref{2-plat 2-knot is ribbon}), their double branched covers do not distinguish among 2-plat 2-knots with the same denominator:

\begin{theorem}[\cite{Montesinos1978}]\label{Fact:Montesinos1978}
    Suppose $S^4 = \R^4 \cup \{\infty\}$ so that 2-knots are embedded in $S^4$.
    Then, double covers of $S^4$ branched along $F(p,q)$ and $F(p',q')$ are diffeomorphic if and only if $p = p'$.
\end{theorem}

Hence, we cannot apply the strategy of classifying 2-bridge 1-knots to oriented 2-plat 2-knots.
To overcome this problem, we introduce new invariants for a 2-knot $F$:
\begin{align*}
    a(F) ~=~ \sum_{k \in \Z} |a_k|, \quad \tau(F) ~=~ \sum_{k \in \Z} k \, |a_k| \pmod{a(F)},
\end{align*}
where $\Delta_F(t) = \sum_{k \in \Z}a_k\, t^k$.
Although the Alexander polynomial $\Delta_F(t)$ is defined only up to multiplication by $\pm t^n$, these are well-defined (\Cref{well-definedness of tau}).

\begin{theorem}\label{Main:Formula of tau invariant}
    Suppose that $p$ is odd and $q$ is even.
    Then, we have
    \begin{align*}
        a(F(p,q)) ~=~ p, \quad \tau(F(p,q)) ~\equiv~ (2q)^{-1} \pmod p.
    \end{align*}
\end{theorem}

The proof is given in \Cref{Section:tau invariant}.
Using this invariant, we establish the following classification:

\begin{theorem}\label{Main:Classification of 2-plat 2-knots}
    Let $F(p, q)$ and $F(p', q')$ be oriented 2-plat 2-knots.
    Then, the following are equivalent:
    \begin{enumerate}
        \item $F(p,q)$ and $F(p',q')$ are equivalent.
        \item $p = p'$ and $q \equiv q' \pmod{p}$ hold.
    \end{enumerate}
\end{theorem}

\begin{proof}
    The case (2) $\Rightarrow$ (1) is obtained by the definition of normal forms and by \Cref{2-plat 2-knots modulo p}.
    Thus, it remains to prove that (1) implies (2).

    Suppose that $F(p,q)$ and $F(p',q')$ are equivalent.
    Then, $p = p'$ holds by \Cref{Determinant of 2-plat 2-knot}.
    Since $F(p,q)$ and $F(p,q+p)$ are equivalent (\Cref{2-plat 2-knots modulo p}), we may assume that $q$ is even.
    By \Cref{Main:Formula of tau invariant}, we have
    \begin{align*}
        (2q)^{-1} \equiv (2q')^{-1} \pmod{p}.
    \end{align*}
    Hence, we have $q \equiv q' \pmod{p}$.
\end{proof}

It follows from the proof of the classification that the Alexander polynomial is a complete invariant for oriented 2-plat 2-knots.
This is somewhat surprising, because the Alexander polynomial is far from being a complete invariant for general ribbon 2-knots or even for 2-bridge 1-knots in view of the following facts:
\begin{itemize}
    \item Many distinct 2-bridge 1-knots share the same Alexander polynomial (\cite{Kanenobu-Sumi1993}).
    \item The Alexander polynomial of any 2-knot is realized by that of some ribbon 2-knot (\cite{Kinoshita1961}, see also \Cref{fact:Kinoshita1961}).
    \item There exist infinitely many distinct ribbon 2-knots with trivial Alexander polynomials (\cite{Kanenobu-Sumi2018}, see also \Cref{fact:Kanenobu-Sumi2018}).
\end{itemize}

A Laurent polynomial $f(t)$ is called \textit{reciprocal} if $f(t) \doteq f(t^{-1})$.
If $F$ has a reciprocal Alexander polynomial, then $\tau(F) \equiv 0$ by definition.
In particular, the invariant $\tau(K)$ is trivial for every 1-knot $K$.
Furthermore, $\tau(F)$ gives an obstruction to the invertibility of 2-knots:

\begin{theorem}
    For an invertible 2-knot $F$, we have $\tau(F) \equiv 0 \pmod{a(F)}$.
\end{theorem}

\begin{proof}
    It follows from the fact that $\Delta_{-F}(t) \doteq \Delta_F(t^{-1})$.
\end{proof}

\begin{corollary}\label{Main:2-plat 2-knot is non-invertible}
    Every non-trivial 2-plat 2-knot is non-invertible.
\end{corollary}

\begin{proof}
    Let $F(p,q)$ be a non-trivial 2-plat 2-knot, that is, with $p > 1$ and $|q| > 1$.
    Since $2q \in \Z/p\Z$ is invertible, we have $\tau(F(p,q)) \equiv (2q)^{-1} \not\equiv 0 \pmod p$ by \Cref{Main:Formula of tau invariant}.
    Thus, $F(p,q)$ is non-invertible.
\end{proof}

We mention that it is not easy to detect the non-reciprocity of $\Delta_{F(p,q)}(t)$ from \Cref{Main: Formula of Alexander polynomials for 2-plat 2-knots} because $\Delta_{F(p,q)}(t)$ is decomposed into monomials in the formula.

Next, we compare 2-plat 2-knots with spun 2-knots (\cite{Artin1925-spun}).
For a 1-knot $K$, let $\sigma K$ denote its spun 2-knot.
The spun 2-knot of a 2-bridge 1-knot $K(p,q)$ has a ribbon presentation as shown in \Cref{fig:Intro-2}, where $q/p = [c_1, \dots, c_m]$.
Thus, $F(p,q)$ and $\sigma K(p,q)$ have similar ribbon presentations.
This leads to the following question:

\begin{question}
    Is a 2-plat 2-knot equivalent to the spun 2-knot of some 2-bridge 1-knot?
\end{question}


\begin{theorem}
    No non-trivial 2-plat 2-knot is equivalent to a spun 2-knot:
    \begin{align*}
        \left\{F(p,q) \mbox{: 2-plat 2-knots}\right\} \cap \{\sigma K\mbox{: spun 2-knots}\} ~=~ \{F(1,0)\mbox{: a trivial 2-knot}\}.
    \end{align*}
\end{theorem}

\begin{proof}
    For a 1-knot $K$, Artin \cite{Artin1925-spun} showed that the fundamental groups of the complements of $K$ and $\sigma K$ are isomorphic.
    Since the Alexander polynomial is determined by the knot group, $\Delta_{\sigma K}(t)$ is reciprocal.
    Hence, $\tau(\sigma K) \equiv 0$ holds.
    Therefore, $\tau(F)$ distinguishes 2-plat 2-knots from spun 2-knots.
\end{proof}

In the study of 2-bridge 1-knots, the parameter $q$ in $K(p,q)$ is also called the \textit{torsion}, in the sense that it is detected by the Reidemeister torsion of $L(p,q)$.
Therefore, our invariant $\tau(F)$ may capture torsion-like information for closed 4-manifolds.

\begin{question}
    What is the geometric meaning of $\tau(F)$?
\end{question}

Finally, we focus on fibered 2-plat 2-knots.
Computer calculations yield the following proposition:
\begin{proposition}\label{Computer calculation of 2-plat 2-knots 2}
    Every 2-plat 2-knot $F(p,q)$ with $1 < p \leq 2000$ has a non-monic Alexander polynomial.
    Here, a monic polynomial means that the coefficients of the highest- and lowest-degree terms are $\pm 1$.
\end{proposition}

A fibered ribbon 2-knot of 1-fusion has a monic Alexander polynomial (\cite{Kanenobu-Sumi2020,Yoshikawa1981}).
Hence, \Cref{Computer calculation of 2-plat 2-knots 2} leads to the following conjecture:

\begin{conjecture}
    Every non-trivial 2-plat 2-knot is non-fibered.
\end{conjecture}

This paper is organized as follows:
In \Cref{Section:Preliminaries}, we recall the notions of plat closures of 2-dimensional braids and ribbon surface-links.
In \Cref{Section: Normal forms of 2-plat 2-knots}, we introduce normal forms of oriented 2-plat 2-knots.
In \Cref{Section: Alexander polynomials of 2-plat 2-knots}, we compute the Alexander polynomials and $\tau(F)$ of oriented 2-plat 2-knots.
In the appendix, we list 2-plat 2-knots $F(p,q)$ with $p \leq 19$ in \Cref{table: 2-plat 2-knots}, together with their ribbon types and normalized Alexander polynomials.

\medskip
\noindent
\textbf{Acknowledgements.}
This work was supported by JSPS KAKENHI Grant Numbers JP22J20494, JP25K23341 and by Research Fellowship Promoting International Collaboration, The Mathematical Society of Japan.

\section{Plat closure of 2-dimensional braids and ribbon surface-links}\label{Section:Preliminaries}
A \textit{surface-knot} is a connected closed surface embedded in $\R^4$, and a \textit{surface-link} is a union of disjoint surface-knots.
An orientable surface-link is called \textit{trivial} if it bounds mutually disjoint handlebodies embedded in $\R^4$.
In this section, we recall the notions of plat closures of 2-dimensional braids, braid systems, and ribbon surface-links.

\subsection{The plat closure of 2-dimensional braids}\label{Subsection:Plat closure of 2-dim braid}
Let $D^2$ be the unit disk in $\R^2$, and let $X_n$ ($n \geq 1$) be a fixed subset of $n$ interior points in $D^2$ lying on the x-axis.
A (\textit{geometric}) \textit{$n$-braid} is a disjoint union $\beta$ of $n$ curves properly embedded in $D^2 \times I$ such that
\begin{enumerate}
    \item[(b1)] the restriction map $\pi_\beta := p_2|_\beta: \beta \to I$ is a covering map of degree $n$, and
    \item[(b2)] $\partial \beta = X_n \times \partial I$,
\end{enumerate}
where $I = [0,1]$ and $p_2: D^2 \times I \to I$ is the projection onto the second factor.
Two $n$-braids are said to be \textit{equivalent} if they are related by an isotopy of $D^2 \times I$ relative to the boundary.
For a geometric braid $\beta$, we sometimes denote its equivalence class by the same symbol $\beta$.
The \textit{braid group} $B_n$ is the group of equivalence classes of $n$-braids.
In \Cref{Subsection: braid system}, we identify $B_n$ with the fundamental group $\pi_1(\mathcal{C}_n, X_n)$ of the configuration space $\mathcal{C}_n$ of $n$ interior points in $D^2$.
See \cite[Section 1]{Kamada2002_book} for details of this identification.

Let $\beta$ be a $2n$-braid, and let $w_n$ be the union of $n$ disjoint line segments in $D^2$ such that $\partial w_n = X_{2n}$.
The \textit{plat closure} of $\beta$, denoted by $\tilde{\beta}$, is the link $\beta \cup (w_n \times \partial I)$ in $\R^2\times \R^1 = \R^3$.
Every link is equivalent to the plat closure of some braid.

A \textit{2-dimensional $n$-braid} (or a \textit{surface $n$-braid}) \cite{Kamada2002_book,Viro90} is a compact surface $S$ properly embedded in a bidisk $D^2_1\times D^2_2$ such that
\begin{enumerate}
    \item[(B1)] the restriction map $\pi_S := \mathrm{pr}_2|_S: S \to D^2_2$ is a simple branched covering map of degree $n$, and
    \item[(B2)] $\partial S = X_n \times \partial D^2_2$,
\end{enumerate}
where $D^2_i$ ($i = 1,2$) are 2-disks in $\R^2$ and  $\mathrm{pr}_i: D^2_1 \times D^2_2 \to D^2_i$ is the projection onto the $i$-th factor ($i = 1, 2$).
Here, a branched covering map $\pi: X \to Y$ of degree $n$ is said to be \textit{simple} if $|\pi^{-1}(y)| \in \{n, n-1 \}$ for each $y \in Y$.
A compact surface $S$ properly embedded in $D^2_1 \times D^2_2$ satisfying (B1) is called a \textit{braided surface} \cite{Rudolph1983}.
Two 2-dimensional braids are said to be \textit{equivalent} if
there exist isotopies $\{H_t\}_{t \in I}$ of $D^2_1 \times D^2_2$ relative to $\partial (D^2_1 \times D^2_2)$ and $\{h_t\}_{t \in I}$ of $D^2_2$ such that $H_0 = \id$, $H_1$ carries one to the other, and $\mathrm{pr}_2\circ H_t = h_t \circ \mathrm{pr}_2$ holds for each $t \in I$.
See \cite{Kamada2002_book, Kamada2017_book} for details.

\begin{definition}
    Let $S$ be a 2-dimensional $2n$-braid and $A = w_n \times \partial D^2_2$.
    The \textit{plat closure} of $S$, denoted by $\tilde{S}$, is the (possibly non-orientable) unoriented surface-link $S \cup A$ in $D^2_1 \times D^2_2 \subset \R^4$.
\end{definition}

\begin{lemma}
    If two 2-dimensional $2n$-braids are equivalent, then their plat closures are equivalent.
\end{lemma}

\begin{proof}
    This follows from the definition of the equivalence for 2-dimensional braids.
\end{proof}

\begin{theorem}[\cite{Yasuda21}]\label{Fact:Plat closure of 2-dim braid}
    Every orientable surface-link is equivalent to the plat closure of some 2-dimensional braid.
\end{theorem}

\begin{remark}
    \Cref{Fact:Plat closure of 2-dim braid} can be generalized to arbitrary (possibly non-orientable) surface-links by using plat closures of braided surfaces (\cite{Yasuda21}).
\end{remark}

If $\pi_S$ has $m$ branch points and the degree of $\pi_S$ is $n$, then the Euler characteristic $\chi(\tilde{S})$ is equal to $2n-m$.

Let $S$ be a 2-dimensional $n$-braid, and let $p_1: \R^4 = \R^3 \times \R^1 \to \R^3$ be the projection onto the first factor $\R^3$.
We identify $D^2_2$ with $I \times [-2,2]$.
For $t \in [-2,2]$, define $S_{[t]} = p_1(S \cap (D^2_1 \times I \times \{t\}))$.
Then, $S_{[t]}$ is an $n$-braid, possibly with singular points.
We call $\{S_{[t]}\}_{t\in [-2,2]}$ a \textit{motion picture} of $S$.
A motion picture of $S$ is illustrated by a sequence of slices $S_{[t]}$ as shown in \Cref{fig:MoPi-2dimBraid}.
The plat closure of $S$ has a motion picture illustrated in \Cref{fig:MoPi-PlatClosure1}.

\begin{figure}[h]
    \centering
    \includegraphics[height = 20mm]{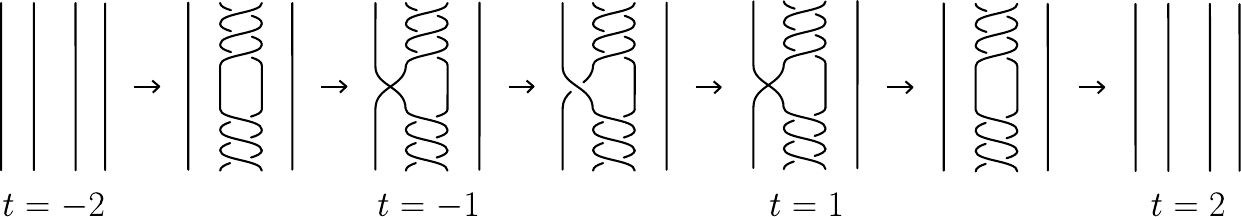}
    \caption{A motion picture of a 2-dimensional 4-braid $S$.}
    \label{fig:MoPi-2dimBraid}
\end{figure}

\begin{figure}[h]
    \centering
    \includegraphics[height = 20mm]{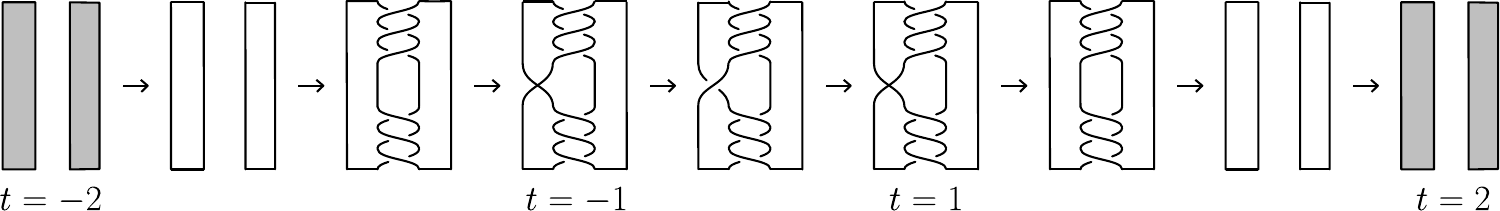}
    \caption{A motion picture of the plat closure $\tilde{S}$.}
    \label{fig:MoPi-PlatClosure1}
\end{figure}

\begin{definition}
    Let $n$ be a positive number.
    A surface-link is said to be \textit{$n$-plat} if it is equivalent to the plat closure of a 2-dimensional $2n$-braid.
\end{definition}

Every 1-plat surface-link is equivalent to a trivial surface-knot (\cite[Theorem~5.5]{Yasuda21}).
Moreover, for each $g \geq 0$ and $n \geq 2$, there exist infinitely many distinct surface-knots of genus $g$ such that they are $n$-plat but are not $(n-1)$-plat (\cite{Yasuda24}).

\begin{remark}
    For a 2-dimensional $2n$-braid $S$, each connected component $F$ of $\tilde{S}$ has the normal Euler number $e(F) = 0$  (\cite{Massey1969,Kamada2017_book}).
    However, the following converse question remains open:
    is every unoriented surface-link, each of whose components $F$ satisfies $e(F)= 0$, equivalent to the plat closure of some 2-dimensional braid?
\end{remark}

\subsection{Braid systems of 2-dimensional braids}\label{Subsection: braid system}

A \textit{braid system} is an ordered $r$-tuple $b = (b_1, \dots, b_r)$ of elements $b_i \in B_n$ such that each $b_i$ is conjugate to $\sigma_1$ or $\sigma_1^{-1}$, where $n \geq 1$, $r \geq 0$, and $\sigma_1, \dots, \sigma_{n-1}$ are Artin's standard generators of $B_n$ \cite{Artin1925-braid}.
The equivalence class of a 2-dimensional braid can be characterized in term of a braid system as follows (\Cref{Braid system2}).

Let $S$ be a 2-dimensional $n$-braid, and let $\Sigma(S) = \{ y_1, \dots, y_r\}$ be the branch locus of $\pi_S$.
Let $\mathcal{C}_n$ be the configuration space of $n$ interior points of $D^2_1$.
Choose a base point $y_0 \in \partial D^2_2$.
For a loop $c: (I,\partial I) \to (D^2_2\setminus \Sigma(S), y_0)$, a loop $\rho_S(c): (I,\partial I) \to (\mathcal{C}_n, X_n)$ is defined as $\rho_S(c)(t) = \mathrm{pr}_1(\pi_S^{-1}(c(t)))$.
Then this map induces a homomorphism $\rho_S: \pi_1(D^2_2\setminus \Sigma(S), y_0) \to B_n = \pi_1(\mathcal{C}_n, X_n)$ by $\rho_S([c]) = [\rho_S(c)]$, which is called the \textit{braid monodromy} of $S$.
See \cite{Kamada2002_book} for details.

A \textit{Hurwitz arc system} in $D^2_2$ (with the base point $y_0$) is an $r$-tuple $\mathcal{A} = (\alpha_1, \cdots, \alpha_r)$ of oriented simple arcs in $D^2_2$ such that
\begin{enumerate}
\item for each $i =1, \ldots, r$, $\alpha_i \cap \partial D^2_2 = \{y_0\}$, where $y_0$ is the terminal point of $\alpha_i$,
\item $\alpha_i \cap \alpha_j = \{y_0\}$ ($i\neq j$), and
\item $\alpha_1, \ldots,  \alpha_r$ appear in this order around the base point $y_0$.
\end{enumerate}
The set of initial points of $\alpha_1, \ldots, \alpha_r$ is called the \textit{starting point set} of $\mathcal{A}$.

Let $\mathcal{A} = (\alpha_1, \cdots, \alpha_r)$ be a Hurwitz arc system in $D^2_2$ with the starting point set $\Sigma(S)$.
Let $N_i$ be a small regular neighborhood of the starting point of $\alpha_i$, and let $\overline{\alpha_i}$ be an oriented arc obtained from $\alpha_i$ by restricting to $D^2_2\setminus \Int N_i$.
For each $i$, let $\gamma_i$ be a loop $\overline{\alpha_i}^{~-1} \cdot \partial N_i \cdot \overline{\alpha_i}$ in $D^2_2\setminus \Sigma(S)$ with a base point $y_0$, where $\partial N_i$ is oriented counter-clockwise.

\begin{definition}
    An $r$-tuple $b_S = (\rho_S([\gamma_1]), \dots, \rho_S([\gamma_r]))$ is called a \textit{braid system of $S$}.
\end{definition}

Since $\partial S$ is the trivial closed braid, we have $\prod_{i=1}^r \rho_S([\gamma_i]) = \rho_S([\partial S]) = 1$, where $1 \in B_n$ is the identity element of $B_n$.
Conversely, every braid system $b = (b_1, \dots, b_r)$ satisfying $\prod_{i=1}^r b_i = 1$ is a braid system of some 2-dimensional braid (\cite{Kamada2002_book,Rudolph1983}).

For a braid system $b = (b_1, \dots, b_r)$ and $i \in \{1,\dots, r-1\}$, we define new braid systems
\[
    b_+ = (b_1, \dots, b_{i-1},~
    b_i \, b_{i+1} \,b_i^{-1},~
    b_i,~
    b_{i+2}, \dots, b_r),
    \quad
    b_- = (b_1, \dots, b_{i-1},~
    b_{i+1},~
    b_{i+1}^{-1} \, b_{i} \,b_{i+1},~
    b_{i+2}, \dots, b_r).
\]
Then, $b_+$ and $b_-$ are said to be obtained from $b$ by a \textit{slide action}.
Two braid systems are \textit{slide equivalent} (or \textit{Hurwitz equivalent}) if one is obtained from the other by a finite sequence of slide actions.

\begin{theorem}[\cite{Kamada2002_book,Rudolph1983}]\label{Braid system2}
    Two 2-dimensional braids are equivalent if and only if their braid systems are slide equivalent.
\end{theorem}

We remark that \Cref{Braid system2} naturally generalizes to (pointed) braided surfaces (\cite{Kamada2002_book,Rudolph1983}).

\subsection{Ribbon surface-links}\label{Subsection:Ribbon 2-knots}
A (\textit{3-dimensional}) \textit{1-handle} $h$ attached to a surface-link $F$ is the image of an embedding $h: D^2 \times I \to \R^4$ such that $h(D^2\times I) \cap F = h(D^2\times \{0,1\})$, where $I = [0,1]$.
For a 1-handle $h$, we define the surface-link $F_h = F \cup \partial h \setminus \mathrm{int}(F \cap h)$, which is said to be obtained by \textit{surgery along $h$}.

When $F$ is oriented, we call $h$ \textit{oriented} if two orientations of $\partial h \cap F = h(D^2 \times \{0\}) \cup h(D^2\times \{1\})$ induced by the orientation of $F$ match on $\partial h$.
For an oriented 1-handle $h$, we assign $F_h$ the orientation induced by the orientation of $F$.

A \textit{2-link} is a union of disjoint 2-knots, and it is called \textit{trivial} if it bounds disjoint 3-balls embedded in $\R^4$.

\begin{definition}
    A \textit{ribbon} surface-link is one obtained from a trivial 2-link by surgery along 1-handles.
    In particular, a ribbon 2-knot is said to be \textit{of $m$-fusion} if it is obtained from an $(m+1)$-component trivial 2-link by surgery along $m$ 1-handles.
\end{definition}

Next, we recall ribbon presentations of ribbon surface-links.
A \textit{band} $b$ attached to a link $L$ is the image of an embedding $b: I \times I \to \R^3$ such that $b(I \times I) \cap L = b(I \times \{0,1\})$.
When $L$ is oriented, we call $b$ \textit{oriented} if two orientations of $b \cap L$ induced by the orientation of $L$ match on $\partial b$.

Let $L$ be a trivial 1-link, and let $B$ be a union of disjoint bands attached to $L$.
Let $D$ be a union of disjoint 2-disks in $\R^3$ such that $\partial D = L$.
Then $F_0 = \partial (D \times [-2,2])$ is a trivial 2-link in $\R^4 = \R^3 \times \R^1$ and $H = B \times [-1,1]$ is a union of disjoint 1-handles attached to $F_0$.
Thus, we obtain the ribbon surface-link $F(L,B)$ by
\begin{align*}
    F(L,B) ~=~ F_0 \cup \partial H \setminus \mathrm{int}( F_0 \cap H).
\end{align*}
When $L$ and $B$ are oriented, we obtain an oriented surface-link $F(L,B)$.
Furthermore, the equivalence class of $F(L, B)$ is independent of the choice of $D$.

We call $(L, B)$ a \textit{ribbon presentation} of a ribbon surface-link $F$ if $F(L, B)$ is equivalent to $F$.
Every ribbon surface-link has a ribbon presentation (\cite{Yajima1964}).

By definition, $F = F(L,B)$ has a motion picture $F_{[t]}$ ($t \in \R$) as shown in \Cref{fig:Example of motion picture of realizing surface} such that
\begin{enumerate}
    \item[(R1)] $F_{[t]}$ contains 2-disks only if $t = \pm 2$. (This implies that $F_{[t]} = \emptyset$ for any $|t| > 2$.)
    \item[(R2)] $F_{[t]}$ contains bands/saddle points only if $t = \pm 1$.
    \item[(R3)] $F_{[t]} = F_{[-t]}$ holds for any $t \in [0,2]$.
\end{enumerate}

\begin{figure}[h]
    \centering
    \includegraphics[width = \hsize]{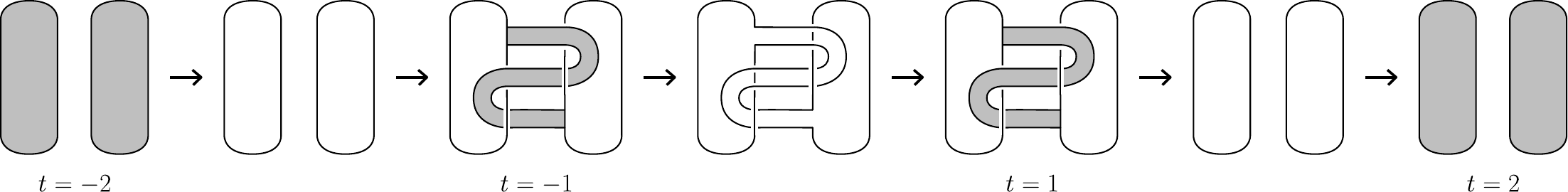}
    \caption{A motion picture of a ribbon surface-link $F(L,B)$.}
    \label{fig:Example of motion picture of realizing surface}
\end{figure}

Conversely, suppose that a surface-link $F$ has a motion picture $F_{[t]}$ ($t \in \R$) satisfying (R1) -- (R3).
Then, $F$ is ribbon (\cite{Kawauchi-Shibuya-Suzuki1982,Kawauchi-Shibuya-Suzuki1983}).
This can be proved as follows:
Let $\mathcal{D}_- = \mathrm{cl}(F \cap (\R^3 \times (-\infty, -1)))$, $\mathcal{A} = \mathrm{cl}(F \cap (\R^3 \times (-1,1)))$, and $\mathcal{D}_+ = \mathrm{cl}(F \cap (\R^3 \times (1,\infty)))$, where $\mathrm{cl}(X)$ is the closure of $X$.
$\mathcal{D}_\pm$ are unions of disjoint 2-disks, $\partial \mathcal{D}_\pm$ are trivial 1-links in $\R^3 \times \{\pm 1\}$, and $\mathcal{A}$ is a union of disjoint annuli.

Let $L = p_1(\partial \mathcal{D}_+)$, $B = F_{[1]} \setminus \mathrm{int}(L)$, and $L_B = p_1(\partial \mathcal{A})$, where $p_1: \R^4 = \R^3 \times \R^1 \to \R^3$ is the projection onto the first factor $\R^3$.
We have $L_B = L \cup \partial B \setminus \mathrm{int}(L \cap B)$.
Take a union $D$ of disjoint 2-disks in $\R^3$ bounded by $L$.
Since $F$ satisfies (R3), $\mathcal{A}$ is isotopic to $L_B \times [-1,1]$ in $\R^3 \times [-1,1]$ relative to the boundary.
Furthermore, $\mathcal{D}_\pm$ are \textit{trivial disk systems} (\cite{Kamada2002_book,Kamada2017_book}); hence $\mathcal{D}_+$ is isotopic to $(L \times [1, 2]) \cup (D \times \{2\})$ in $\R^3 \times [1,\infty]$ relative to the boundary.
Similarly, $\mathcal{D}_-$ is isotopic to $(L \times [-2,-1]) \cup (D \times \{-2\})$ in $\R^3 \times [-\infty, -1]$ relative to the boundary.
As a result, $F$ is isotopic to $F(L,B)$ in $\R^4$.
Therefore, $F$ is ribbon, and $(L, B)$ is a ribbon presentation of $F$.

The reflection of $\R^4$ about $\R^3 \times \{0\}$ preserves $F(L, B)$.
Hence, every ribbon surface-link is negatively amphicheiral; that is, the mirror image $F!$ is equivalent to $-F$.
See \cite{Marumoto1992,Yajima1964,Yanagawa1969} for details on ribbon surface-links.

\begin{definition}[\cite{Kamada1998,Kamada2002_book}]
    A braid system $b = (b_1, \dots, b_r)$ is called \textit{ribbon} if $r$ is even and $b_{2i-1} = b_{2i}^{-1}$ holds for each $i = 1, 2, \dots, r/2$.
    A 2-dimensional braid is called \textit{ribbon} if it has a ribbon braid system.
\end{definition}

\begin{proposition}\label{Ribbon 2-dim braid gives ribbon surface-link}
    Let $S$ be a ribbon 2-dimensional $2n$-braid.
    Then, the plat closure $\tilde{S}$ is a ribbon surface-link.
\end{proposition}

\begin{proof}
    Let $S$ be a ribbon 2-dimensional $2n$-braid and let $b_S =  (b_1, \dots, b_{2r})$ be a ribbon braid system of $S$; $b_{2i-1} = b_{2i}^{-1}$ ($i = 1, \dots, r$).
    For each $i$, take a geometric $2n$-braid $\beta_i$ satisfying $b_{2i-1} = \beta_i \sigma_1 \beta_i^{\, -1}$.
    Then, $S$ has a motion picture as shown in \Cref{fig:MotPic-ribbonBraid}, which is symmetric with respect to $\R^3\times \{0\}$ (\cite[Section~29]{Kamada2002_book}).
    Thus, $\tilde{S}$ also has a symmetric motion picture with respect to $\R^3\times \{0\}$ that satisfies (R1) -- (R3).
    Hence, $\tilde{S}$ is a ribbon surface-link.
\end{proof}

\begin{figure}[h]
    \centering
    \includegraphics[width = \hsize]{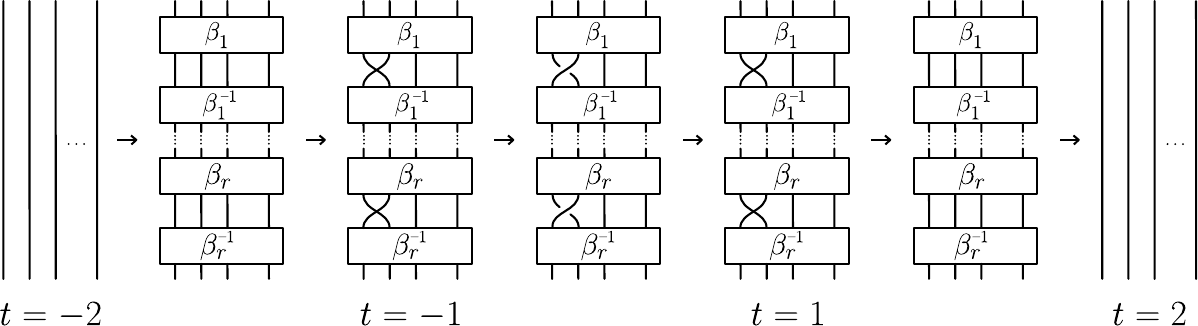}
    \caption{A motion picture $\{S_{[t]}\}$ of a ribbon 2-dimensional braid $S$.}
    \label{fig:MotPic-ribbonBraid}
\end{figure}

\begin{corollary}\label{Ribbon presentation of plat closure}
    Let $S$ be a ribbon $2$-dimensional $2n$-braid with a motion picture $\{S_{[t]}\}_{t\in[-2,2]}$, as illustrated in \Cref{fig:MotPic-ribbonBraid}.
    Then $\tilde S$ has a ribbon presentation obtained from $\tilde S_{[1]}$ by replacing each branch point with a band, as in \Cref{fig:BranchPointToBand}.
\end{corollary}

\begin{figure}[h]
    \centering
    \includegraphics[height = 35mm]{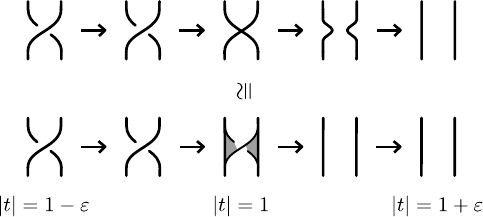}
    \caption{An isotopy deforming a branch point into a band.}
    \label{fig:BranchPointToBand}
\end{figure}

\begin{corollary}[cf. \cite{Yasuda21}]\label{2-plat 2-knot is ribbon}
    For $n\geq 2$, an $n$-plat surface-link $F$ with $\chi(F) = 2n-2$ is ribbon.
    In particular, a 2-plat 2-knot is a ribbon 2-knot of 1-fusion.
\end{corollary}

\begin{proof}
    Let $F$ be an $n$-plat surface-link with $\chi(F) = 2n-2$ and let $S$ be a 2-dimensional $2n$-braid such that $\tilde{S}$ is equivalent to $F$.
    Since $\chi(F) = 2n-2$, the projection $\pi_S$ has exactly two branch points.

    Let $b_S = (b_1, b_2)$ be a braid system of $S$.
    Since $b_S$ satisfies $b_1 b_2 = 1$, $b_2$ is the inverse of $b_1$.
    Therefore, $b_S$ is slide equivalent to a ribbon braid system.
    Thus, $S$ is also a ribbon 2-dimensional braid, and $F$ is a ribbon surface-link by \Cref{Ribbon 2-dim braid gives ribbon surface-link}.
    In particular, the pair of branch points corresponds to a 1-handle; hence $F$ is of 1-fusion.
\end{proof}

\section{Normal forms of 2-plat 2-knots}\label{Section: Normal forms of 2-plat 2-knots}

In this section, we introduce normal forms of oriented 2-plat 2-knots.
Since every 2-bridge 1-knot is invertible, the choice of orientations is not important for their classification.
However, as mentioned in \Cref{Main:2-plat 2-knot is non-invertible}, non-trivial 2-plat 2-knots are non-invertible.
Thus, we should be careful about the choice of orientations.

We first consider the choice of orientation of a 2-bridge 1-knot via Conway's normal forms $K_C(c_1, \dots, c_m)$ and Schubert's normal forms $K_S(p,q)$.
Later, we generalize these normal forms for oriented 2-plat 2-knots.


\subsection{Conway's normal forms of unoriented 2-bridge 1-knots}

Let $[c_1, \dots, c_m]$ denote a continued fraction:
\begin{align*}
    [c_1, c_2, \dots, c_m] ~=~
    \begin{cases}
        1/\left(c_1 + [c_2, c_3, \dots, c_m]\right) & (m \geq 2)\\
        1/c_1  & (m = 1),
    \end{cases}
\end{align*}
where $m > 0$ and $c_i \in \Z$ ($i = 1, 2, \dots, m$).
Here, we use $\infty = 1/0$ as a formal expression.
It is straightforward that $[0, c_2, \dots, c_m] = 1/[c_2, \dots, c_m]$.

A \textit{tangle} is a union of disjoint curves properly embedded in $\R^2 \times (-\infty, 1]$, and two tangles are said to be \textit{equivalent} if they are related by an isotopy of $\R^2\times (-\infty, 1]$ fixing $\R^2\times \{1\}$ pointwise.
A tangle is called \textit{trivial} if it is equivalent to a tangle $T$ such that the restriction $p_2|_T$ has no maximal points, where $p_2: \R^3 = \R^2 \times \R^1 \to \R^1$ is the projection onto the second factor.

Suppose that $m$ is odd.
Let $\beta_c$ be a geometric 4-braid representing $\sigma_2^{c_1} \sigma_1^{-c_2} \sigma_2^{c_3} \cdots \sigma_1^{-c_{m-1}} \sigma_2^{c_m} \in B_4$.
A \textit{rational tangle} $T(c_1, \dots, c_m)$ is a 2-strand tangle defined as a union of a geometric 4-braid $\beta_c$ and $w_2 \times \{0\}$.
\Cref{fig:Dig-RationalTangle} shows a diagram of $T(c_1, \dots, c_m)$.
It was shown in \cite{Conway1970,Kauffman-Lambropoulou2004} that
\begin{itemize}
    \item a 2-strand tangle is trivial if and only if it is equivalent to a rational tangle, and
    \item $T(c_1, \dots, c_m)$ and $T(d_1, \dots, d_{m'})$ are equivalent if and only if $[c_1, \dots, c_m] = [d_1, \dots, d_{m'}]$.
\end{itemize}
Thus, for coprime integers $p, q$ with $p \geq 0$, we define $T(p,q) = T(c_1, \dots, c_m)$, where $q/p = [c_1, \dots, c_m]$.

We define the unoriented 1-knot or 1-link $K_C(c_1, \dots, c_m)$ as the plat closure of $\beta_c$.
We call $K_C(c_1, \dots, c_m)$ \textit{Conway's normal form} for unoriented 2-bridge 1-knots.
Using the above facts on rational tangles, we see that
\begin{itemize}
    \item every unoriented 2-bridge 1-knot is equivalent to $K_C(c_1, \dots, c_m)$ for some $c_i \in \Z$, and
    \item $K_C(c_1, \dots, c_m)$ and $K_C(d_1, \dots, d_n)$ are equivalent if $[c_1, \dots, c_m] = [d_1, \dots, d_n]$.
\end{itemize}

\begin{figure}[h]
    \centering
    \begin{minipage}{0.50\hsize}
        \centering
        \includegraphics[height = 35mm]{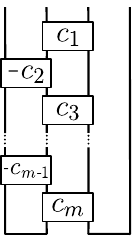}
    \end{minipage}
    \begin{minipage}{0.49\hsize}
        \centering
        \includegraphics[height = 35mm]{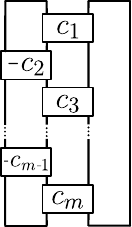}
    \end{minipage}
    \caption{A rational tangle $T(c_1, \dots, c_m)$ and Conway's normal form $K_C(c_1, \dots, c_m)$.}
    \label{fig:Dig-RationalTangle}
\end{figure}

\subsection{Schubert's normal forms of oriented 2-bridge 1-knots}

Let $p$ and $q$ be coprime integers with $0 < |q| < p$, where $p > 0$ is odd.
\textit{Schubert's normal form} $K(p,q)$ of an oriented 2-bridge 1-knot is defined by the following algorithm.

Set $X_4=\{A_1,A_2,A_3,A_4\}$, where each point $A_i=(x_i,0) \in \R^2$ lies on the $x$-axis and $x_i<x_{i+1}$.
Draw the line segments $A_1A_2$ and $A_3A_4$.
We take $p-1$ points on $A_1A_2$ (or $A_3A_4$) at equal intervals, and label the points from $A_1$ (or $A_4$) by $1,2,\dots,p$, respectively.
Let $n_i = \lfloor iq/p \rfloor$ be the integer part of $iq/p$ for $i = 1, \dots, p-1$.
Note that $n_1 = 0$ if $q > 0$, whereas $n_1 = -1$ if $q < 0$.

We draw an under-path $\ell$ of $K(p,q)$ as follows:
$\ell$ starts at $A_1$.
If $n_1$ is even, $\ell$ first passes the point labeled $q$ on $A_3A_4$ by a counter-clockwise $\pi$-rotation.
Otherwise, if $n_1$ is odd, $\ell$ first passes the point labeled $-q$ on $A_3A_4$ by a clockwise $\pi$-rotation.
Next, if $n_2$ is even, $\ell$ passes the point labeled $2q \pmod{p}$ on $A_1A_2$ in the counter-clockwise direction.
Otherwise, if $n_2$ is odd, it passes the point labeled $-2q \pmod{p}$ on $A_1A_2$ in the clockwise direction.
Repeating this procedure, $\ell$ alternately passes points on $A_1A_2$ and $A_3A_4$ labeled $iq \pmod p$ or $-iq \pmod p$ in the counter-clockwise or clockwise direction according to the parity of $n_i$.
After passing $p-1$ points, $\ell$ arrives at $A_3$ if $n_p = q$ is odd, otherwise arrives at $A_4$ if $n_p = q$ is even.
Drawing the symmetric under-path, we obtain another under-path $\ell'$.
We define the unoriented 1-knot $K(p,q) = w_2 \times \{1\} \cup (\ell \cup \ell') \times \{0\}$, where $w_2 = A_1A_2 \cup A_3A_4$.
We call $w_2 \times \{1\}$ the \textit{bridges} of $K(p,q)$.

Finally, we choose an orientation on the line segment $A_3A_4 \times \{1\}$ pointing toward $A_3$.
This induces an orientation of $K(p,q)$, thereby making it an oriented 1-knot $K(p,q)$.
We call it the \textit{canonical orientation} of $K(p,q)$.
See \Cref{fig:NF-2bridgeKnot}.
By definition, the orientation on $A_1A_2 \times \{1\}$ points toward $A_2$ if $q$ is odd, otherwise points toward $A_1$ if $q$ is even.


\begin{figure}[h]
    \centering
    \begin{minipage}{0.47\hsize}
        \centering
        \includegraphics[width = 0.6\hsize]{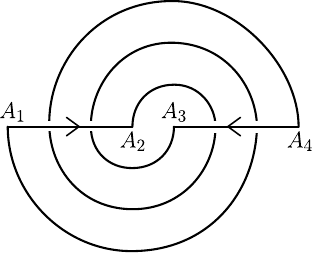}
    \end{minipage}
    \begin{minipage}{0.47\hsize}
        \centering
        \includegraphics[width = 0.6\hsize]{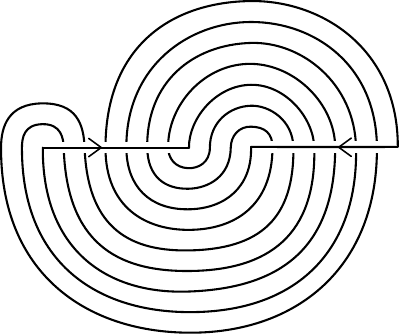}
    \end{minipage}
    \caption{Schubert's normal forms of $K(3,1)$ and $K(7,3)$.}
    \label{fig:NF-2bridgeKnot}
\end{figure}

\begin{remark}
    In this paper, we use the symbol $K_C(c_1, \dots, c_m)$ for unoriented 2-bridge 1-knots and $K(p,q)$ for oriented ones.
\end{remark}

The following two propositions are obtained by reading the signs of crossings along an under-path of $K(p,q)$:

\begin{proposition}[\cite{Schubert1954}]\label{Presentation of knot group of 2-bridge knots}
    Suppose that both $p$ and $q$ are odd, and put $\varepsilon_i = (-1)^{\lfloor iq/p \rfloor}$.
    Then, the fundamental group $\pi_1(\R^3 \setminus K(p,q))$ has a presentation
    \[
        \langle x, y \mid y = w x w^{-1} \rangle, \quad
        w = x^{\varepsilon_1}y^{\varepsilon_2}x^{\varepsilon_3}y^{\varepsilon_4}\cdots x^{\varepsilon_{p-2}}y^{\varepsilon_{p-1}}.
    \]
    Here, $x$ and $y$ represent meridional loops of the two bridges.
\end{proposition}

\begin{proposition}[\cite{Hartley1979,Minkus1982}]\label{Alexander polynomial of 2-bridge knots}
    Suppose that both $p$ and $q$ are odd, and put $\varepsilon_i = (-1)^{\lfloor iq/p \rfloor}$ for $i = 1, 2, \dots, p-1$.
    Then, we have
    \begin{align*}
        \Delta_{K(p,q)}(t)  ~\doteq~ 1 + \sum_{k=1}^{p-1} (-1)^k t^{\,d(k)} ~=~ 1 - t^{\,\varepsilon_1} - t^{\,\varepsilon_1 + \varepsilon_2} + \cdots + t^{\,\varepsilon_1 + \cdots + \varepsilon_{p-1}}, \quad d(k) = \sum_{i=1}^{k} \varepsilon_i,
    \end{align*}
\end{proposition}

We note that Hoste \cite{Hoste2020} gave a similar formula for 2-variable Alexander polynomials of 2-bridge 1-links.

\begin{theorem}[\cite{Schubert1954}]
    $K(p,q)$ and $K(p',q')$ are equivalent if and only if the following hold: (1) $p = p'$ and (2) $q' \equiv q^{\pm 1} \pmod{p}$.
\end{theorem}


Next, we consider oriented 2-bridge 1-knots in Conway's normal forms.


\begin{definition}[\cite{Kauffman-Lambropoulou2004}]
    A continued fraction $q/p = [c_1, \dots, c_m]$ is called \textit{canonical} if $m$ is odd and it satisfies one of the following conditions:
    \begin{enumerate}
        \item $[c_1, \dots, c_m] = [0, 1, -1]$. (In this case, $q/p = 0/1$.)
        \item If $c_1 \neq 0$, all $c_i$ ($i = 1, \dots, m$) are non-zero and have the same sign. (In this case, $|q| < p$ holds.)
        \item If $c_1 = 0$, all $c_i$ ($i = 2, \dots, m$)  are non-zero and have the same sign. (In this case, $p < |q|$ holds.)
    \end{enumerate}
\end{definition}

For any coprime non-zero integers $p, q$, the fraction $q/p$ has a canonical continued fraction expansion by the Euclidean algorithm.
Furthermore, such an expansion is unique for each $q/p$ (\cite[Proposition~3]{Kauffman-Lambropoulou2004}).
The canonical continued fraction expansions of $0/1$ and $1/0$ are $[0,1,-1]$ and $[0]$, respectively.

Let $f_i:\mathbb{R}^2\to\mathbb{R}^2$ be a half-twist around the line segment $A_iA_{i+1}$ counter-clockwise. See \Cref{fig:HalfTwist2}.
Let $[c_1, \dots, c_m]$ be the canonical continued fraction expansion of $q/p$.
We define the homeomorphism $f_c: \R^2 \to \R^2$ by $f_c = f_2^{\, c_1} \circ f_1^{\, -c_2} \circ f_2^{\, c_3} \circ \cdots \circ f_1^{\, -c_{m-1}} \circ f_2^{\, c_m}$.
Then, $f_c(w_2)$ is a union of under-paths of $K(p,q)$.

\begin{figure}[h]
    \centering
    \includegraphics[height = 30mm]{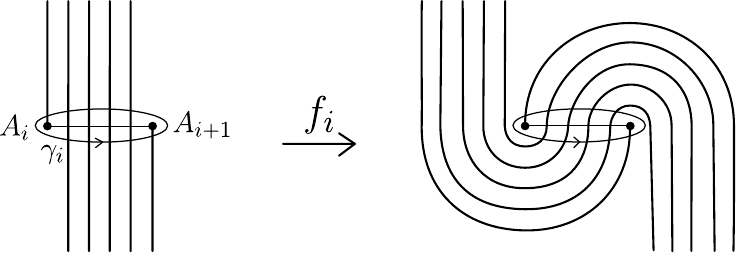}
    \caption{The positive half-twist $f_i$ around $A_iA_{i+1}$.}
    \label{fig:HalfTwist2}
\end{figure}

Let $\{F_{i,t}\}_{t \in I}$ be an isotopy of $\R^2$ such that $F_{i,0} = \id_{\R^2}$ and $F_{i,1} = f_i$.
We define two isotopies $\{F_t: \R^2 \to \R^2\}_{t \in I}$ and $\{\varPhi_t: \R^3 \to \R^3\}_{t \in I}$ by
\begin{align*}
    F_t &= F_{2,t}^{\, c_1} \circ F_{1,t}^{\, -c_2} \circ F_{2,t}^{\, c_3} \circ \cdots \circ F_{1,t}^{\, -c_{m-1}} \circ F_{2,t}^{\, c_m} \quad (t \in I),\\
    \varPhi_t(x,s) ~&=~
    \begin{cases}
        (F_{(t - t|s|)}(x), s) & |s| \leq 1,\\
        (x,s) & |s| \geq 1,
    \end{cases} \quad \left( (x,s) \in \R^2 \times \R^1 = \R^3, t \in I \right).
\end{align*}
By definition, we have $\varPhi_0 = \id_{\R^3}$, and $\{F_t\}$ is a planar isotopy carrying $w_2$ to $f_c(w_2)$.
Moreover, the geometric 4-braid $\beta_c$ defined by
\begin{align*}
    \beta_c ~=~ \bigcup_{t \in I} \left( F_{1-t}(X_4) \times \{t\} \right) ~\subset~ \R^2 \times [0,1]
\end{align*}
represents $\sigma_2^{c_1} \sigma_1^{-c_2} \sigma_2^{c_3} \cdots \sigma_1^{-c_{m-1}} \sigma_2^{c_m} \in B_4$.
Therefore, we have the following proposition:

\begin{proposition}\label{Relationship between Schubert and Conway normal forms}
    Let $[c_1, \dots, c_m]$ be the canonical continued fraction expansion of $q/p$, and let $T$ be the 2-strand tangle obtained by removing the bridges from $K(p,q)$.
    Then, $T$ is equivalent to a rational tangle $T(c_1, \dots, c_m)$.
\end{proposition}

\begin{proof}
    Let $f_c$, $\beta_c$, $\{\varPhi_t\}_{t \in I}$ be as above.
    Since $\beta_c$ represents $\sigma_2^{c_1} \sigma_1^{-c_2} \sigma_2^{c_3} \cdots \sigma_1^{-c_{m-1}} \sigma_2^{c_m} \in B_4$, $T(c_1, \dots, c_m) = \beta_c \cup w_2 \times \{0\}$.
    Moreover, $f_c(w_2)$ is a union of under-paths of $K(p,q)$ so that $T = (\partial w_2 \times [0,1]) \cup (f_c(w_2) \times \{0\})$.
    Therefore, $T$ is equivalent to $T(c_1, \dots, c_m)$ via the isotopy $\{\varPhi_t\}_{t \in I}$.
\end{proof}

The isotopy $\{\varPhi_t\}_{t \in I}$ of $\R^3$ in the proof above fixes the bridges of $K(p,q)$ pointwise.
Hence, the orientations of 2-bridge 1-knots shown in \Cref{fig:Dig-2bridgeKnot_Conway} are well-defined.

\begin{figure}[h]
    \begin{minipage}{0.45\hsize}
        \centering
        \includegraphics[height = 40mm]{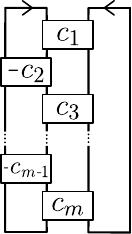}
    \end{minipage}
    \quad
    \begin{minipage}{0.45\hsize}
        \centering
        \includegraphics[height = 40mm]{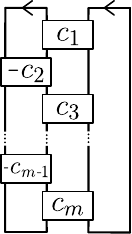}
    \end{minipage}
    \caption{An oriented 2-bridge 1-knot $K(p,q)$ with $q$ odd (left) and with $q$ even (right).}
    \label{fig:Dig-2bridgeKnot_Conway}
\end{figure}

\subsection{Normal forms of 2-plat 2-knots}\label{Subsection: normal forms of 2-plat 2-knots}

First, we consider unoriented 2-plat 2-knots.
Let $m \geq 1$ be an odd number, and let $c_1, \dots, c_m$ be integers.
Let $S(c_1, \dots, c_m)$ be a ribbon 2-dimensional 4-braid with a ribbon braid system $(\beta_c \sigma_1 \beta_c^{-1}, \beta_c \sigma_1^{\,-1} \beta_c^{-1})$, where $\beta_c$ is a geometric 4-braid representing $\sigma_2^{c_1} \sigma_1^{-c_2} \sigma_2^{c_3} \dots \sigma_1^{-c_{m-1}} \sigma_2^{c_m}\in B_4$.

By \Cref{Ribbon 2-dim braid gives ribbon surface-link}, the plat closure $F = \tilde{S}(c_1, \dots, c_m)$ is an unoriented ribbon surface-link.
By \Cref{Ribbon presentation of plat closure}, it has a ribbon presentation $(L, B)$ as shown in \Cref{fig:Intro-1}, where $L$ is a 2-component 1-link and $B$ is a single band.
Put $q/p = [c_1, \dots, c_m]$.
By comparison with the case of 2-bridge 1-knots, the band $B$ is attached to different components of $L$ if and only if $p$ is odd.
Since the Euler characteristic $\chi(F) = 2$, it follows that $F$ is a 2-knot if $p$ is odd; otherwise, $F$ consists of a 2-sphere and a Klein bottle.

\begin{lemma}\label{easy result}
    For an integer $k$, $\tilde S(c_1, \dots, c_m)$ and $\tilde S(0, k, c_1, \dots, c_m)$ are equivalent.
    Similarly, $\tilde S(c_1, \dots, c_m)$ and $\tilde S(c_1, \dots, c_m, k, 0)$ are equivalent.
\end{lemma}

\begin{proof}
    This follows immediately from ribbon presentations as shown in \Cref{fig:Intro-1}.
    See also \Cref{fig:MainProof2}.
\end{proof}

\begin{lemma}\label{Well-definedness of normal form for 2-plat 2-knots}
    Let $m,n >0$ be odd integers and let $c_i, d_j \in \Z$ ($i = 1,\dots, m, j = 1, \dots, n $).
    Then, $\tilde S(c_1, \dots, c_m)$ and $\tilde S(d_1, \dots, d_n)$ are equivalent if $[c_1, \dots, c_m] = [d_1, \dots, d_n]$.
\end{lemma}

\begin{proof}
    Suppose that $[c_1, \dots, c_m] = [d_1, \dots, d_{n}] = q/p$.
    Put $b/a = [c_m, \dots, c_1]$ and $b'/a' = [d_n, \dots, d_1]$.
    By elementary computations of continued fractions, we have $a = a' = p$ and $b \equiv b' \pmod{a}$.
    Hence, we have $b = b' + ka$ for some $k \in \Z$, which implies that $b/a = [0, k, d_n, \dots, d_1]$.
    By \Cref{easy result}, $\tilde S(d_1, \dots, d_n)$ is equivalent to $\tilde S(d_1, \dots, d_n, k, 0)$.
    Thus, it suffices to show that $\tilde S(c_1, \dots, c_m)$ and $\tilde S(d_1, \dots, d_n, k, 0)$ are equivalent.

    Let $(L, B)$ and $(L', B')$ be ribbon presentations of $\tilde S(c_1, \dots, c_m)$ and $\tilde S(d_1, \dots, d_n, k, 0)$ as shown in \Cref{fig:Intro-1,fig:MainProof2}, respectively.
    We divide $(L', B')$ into three parts as shown on the right-hand side of \Cref{fig:MainProof2}: a rational tangle $T' = T(0, k, d_n, \dots, d_1)$, four parallel strands with a single band, and the mirror image of $T'$.
    Similarly, we divide $(L, B)$ into three parts: a rational tangle $T = T(c_m, \dots, c_1)$, four parallel strands with a single band, and the mirror image of $T$.
    Since $[c_m, \dots, c_1] = [0, k, d_n, \dots, d_1] = b/a$, $T$ and $T'$ are equivalent.
    Hence, $(L, B)$ is isotopic to $(L', B')$.
    Therefore, two surface-links $\tilde{S}(c_1, \dots, c_m)$ and $\tilde{S}(d_1, \dots, d_n)$ are equivalent.
    \begin{figure}[h]
        \centering
        \includegraphics[width = 0.8\hsize]{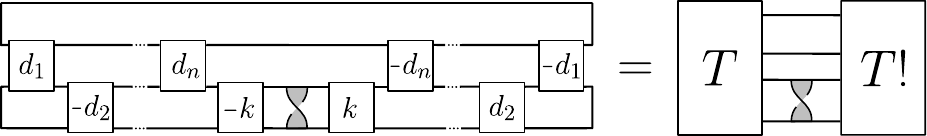}
        \caption{A ribbon presentation $(L', B')$ of $\tilde S(d_1, \dots, d_n, k, 0)$ and its decomposition.}
        \label{fig:MainProof2}
    \end{figure}
\end{proof}

Next, we consider oriented $2$-plat $2$-knots.
To choose a canonical orientation of the plat closure $\tilde{S}(c_1,\dots,c_m)$, we use the canonical continued fraction expansion of $q/p$.

\begin{definition}\label{def: oriented 2-plat 2-knot}
    Let $p$ and $q$ be coprime integers with $p > 0$.
    Let $[c_1,\dots,c_m]$ be the canonical continued fraction expansion of $q/p$.
    We define $F(p,q)$ to be $\tilde{S}(c_1,\dots,c_m)$ equipped with the orientation indicated in \Cref{fig:normal form of 2-plat 2-knot with a odd} if $q$ is odd, and in \Cref{fig:normal form of 2-plat 2-knot with a even} if $q$ is even.
\end{definition}

Since the band has a half-twist, we see that the orientation of one component of $L$ is reversed compared with the orientation of $K(p,q)$.

\begin{figure}[h]
    \begin{minipage}{0.45\hsize}
        \centering
        \includegraphics[width = \hsize]{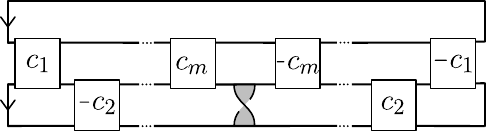}
        \caption{A normal form of $F(p,q)$ with $q$ odd.}
        \label{fig:normal form of 2-plat 2-knot with a odd}
    \end{minipage}
    \quad
    \begin{minipage}{0.45\hsize}
        \centering
        \includegraphics[width = \hsize]{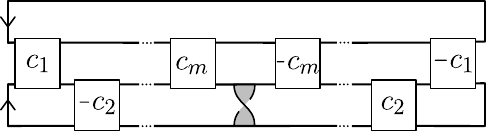}
        \caption{A normal form of $F(p,q)$ with $q$ even.}
        \label{fig:normal form of 2-plat 2-knot with a even}
    \end{minipage}
\end{figure}

\begin{remark}
    In this paper, we use the symbol $\tilde{S}(c_1, \dots, c_m)$ for unoriented 2-plat 2-knots and $F(p,q)$ for oriented ones.
\end{remark}


\begin{lemma}\label{2-plat 2-knots modulo p}
    $F(p,q)$ and $F(p,q')$ are equivalent if $q \equiv q' \pmod{p}$.
\end{lemma}

\begin{proof}
    It suffices to show that $F(p,q)$ and $F(p,q+p)$ are equivalent.
    Let $[c_1,\dots,c_m]$ be the canonical continued fraction expansion of $q/p$, and let $q' = q + p$.
    We remark that $q$ and $q'$ have opposite parity.

    First, we consider the case $p < |q|$.
    Then, we have $c_1 = 0$.
    Hence, $F(p,q)$ has a ribbon presentation as shown in \Cref{fig:NF-2plat2KnotGeneral}.
    Thus, in this case, we see that $F(p,q)$ and $F(p,q')$ are equivalent by comparing their ribbon presentations.
    Similarly, $F(p,q)$ and $F(p,q')$ are equivalent if $p < |q'|$ holds.

    \begin{figure}[h]
        \centering
        \includegraphics[height = 17mm]{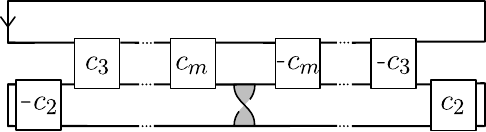}
        \caption{A ribbon presentation of $F(p,q)$ with $p < |q|$.}
        \label{fig:NF-2plat2KnotGeneral}
    \end{figure}

    Therefore, we may consider the case $|q|, |q'| < p$.
    Since $q' = q + p$, we have $-p < q < 0 < q' < p$.
    Then, the canonical continued fraction expansion of $q'/p$ is described in terms of the $c_i$ as follows.
    \begin{enumerate}
        \item If $c_1 \geq 2$ and $c_m \geq 2$, then
        \[
        q'/p=[1,\ 1-c_1,\ -c_2,\ \dots,\ -c_{m-1},\ 1-c_m,\ 1].
        \]
        \item If $c_1 \geq 2$ and $c_m = 1$, then
        \[
        q'/p=[1,\ 1-c_1,\ -c_2,\ \dots,\ -c_{m-2},\ -c_{m-1}-1].
        \]
        \item If $c_1 = 1$ and $c_m \geq 2$, then
        \[
        q'/p=[-c_2-1,\ -c_3,\ \dots,\ -c_{m-1},\ 1-c_m,\ 1].
        \]
        \item If $c_1 = c_m = 1$, then
        \[
        q'/p=[-c_2-1,\ -c_3,\ \dots,\ -c_{m-2},\ -c_{m-1}-1].
        \]
    \end{enumerate}
    In each case, it is straightforward to check that the ribbon presentation of $F(p,q')$ given in \Cref{fig:normal form of 2-plat 2-knot with a odd} or \Cref{fig:normal form of 2-plat 2-knot with a even} is isotopic to that of $F(p,q)$.
    For instance, \Cref{fig:RP-2plat2Knot4} is a ribbon presentation of $F(p,q')$ in case (1) with $q$ even.
    The other cases are similar, and we omit the details.
\end{proof}

\begin{figure}[h]
    \centering
    \includegraphics[height = 17mm]{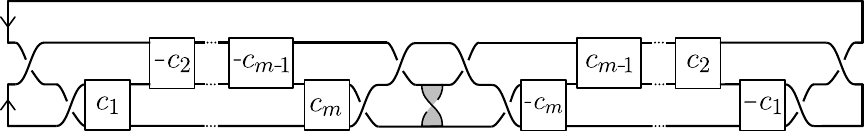}
    \caption{A ribbon presentation of $F(p, q+p)$ with $q$ even.}
    \label{fig:RP-2plat2Knot4}
\end{figure}

\begin{lemma}\label{Mirror image of 2-plat 2-knots}
    $F(p,-q)$ and $-F(p,q)$ are equivalent.
\end{lemma}

\begin{proof}
    Let $(L, B)$ be a ribbon presentation of $F(p,q)$ as shown in \Cref{fig:Intro-1}.
    Let $(L', B')$ be the mirror image of $(L, B)$.
    Then, the sign of each crossing is changed, so $(L',B')$ is a ribbon presentation of $F(p,-q)$.
    Since every ribbon surface-link is negatively amphicheiral, $F(p,-q)$ and $-F(p,q)$ are equivalent.
\end{proof}

\begin{proposition}\label{oriented 2-plat 2-knot has a normal form}
    Every oriented 2-plat 2-knot is equivalent, as an oriented 2-knot, to $F(p,q)$ for some coprime $p,q \in \Z$ with $p > 0$ odd.
\end{proposition}

\begin{proof}
    Let $F$ be an (oriented) 2-plat 2-knot and let $S$ be a 2-dimensional 4-braid such that $F$ is equivalent to $\tilde{S}$ as unoriented 2-knots.
    By \Cref{Ribbon presentation of plat closure}, $F$ has a ribbon presentation as shown on the right-hand side of \Cref{fig:MainProof1} for some geometric 4-braid $\beta$.
    We divide this ribbon presentation into three parts: a 2-strand trivial tangle $T$, four parallel strands with a single band, and the mirror image of $T$.
    See the left-hand side of \Cref{fig:MainProof1}.
    \begin{figure}[h]
        \centering
        \includegraphics[height = 17mm]{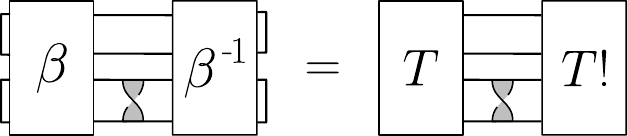}
        \caption{A decomposition of a ribbon presentation.}
        \label{fig:MainProof1}
    \end{figure}
    Since a 2-strand trivial tangle $T$ is equivalent to a rational tangle $T(p,q)$ for some coprime $p,q \in \Z$ with $p > 0$, $F$ is equivalent to a 2-plat 2-knot $F(p,q)$ as unoriented 2-knots.
    In other words, $F$ is equivalent to $F(p,q)$ or $-F(p,q)$ as oriented 2-knots.
    Since $-F(p,q)$ and $F(p,-q)$ are equivalent by \Cref{Mirror image of 2-plat 2-knots}, $F$ is equivalent to $F(p,q)$ for some coprime $p,q \in \Z$ with $p > 0$.
    Since $F$ is a 2-knot, $p$ must be odd.
\end{proof}

\begin{proof}[Proof of \Cref{Main: normal form for 2-plat 2-knot}]
    By definition, $F(p,q)$ is the plat closure of a 2-dimensional 4-braid.
    Thus, $F(p,q)$ is an oriented 2-plat 2-knot.
    Conversely, every oriented 2-plat 2-knot is equivalent to some $F(p,q)$ with $p$ odd by \Cref{oriented 2-plat 2-knot has a normal form}.
\end{proof}

\section{Alexander polynomials of 2-plat 2-knots and the invariant $\tau(F)$}\label{Section: Alexander polynomials of 2-plat 2-knots}

In this section, we determine the ribbon types of 2-plat 2-knots (\Cref{Ribbon types of 2-plat 2-knots}), and we use them to compute their Alexander polynomials (\Cref{Main: Formula of Alexander polynomials for 2-plat 2-knots}) and the invariant $\tau(F)$ (\Cref{Main:Formula of tau invariant}).

\subsection{Ribbon types of ribbon 2-knots of 1-fusion}\label{Subsection: Ribbon types of ribbon 2-knots of 1-fusion}
Any oriented ribbon 2-knot $F$ of 1-fusion can be represented by a finite sequence of integers $p_1, q_1, p_2, q_2, \dots, p_n, q_n$ ($n \geq 1$); such a sequence is called a \textit{ribbon type} and denoted by $R(p_1, q_1, \dots, p_n, q_n)$.
We now recall the definition of ribbon types due to Marumoto \cite{Marumoto1977}.

Let $(L, B)$ be a ribbon presentation of $F$ consisting of an oriented 2-component link $L = L_0 \cup L_1$ and a single oriented band $B$.
Let $D = D_0 \cup D_1$ be a union of disjoint 2-disks in $\R^3$ such that $\partial D_i = L_i$ ($i = 0, 1$) and $D$ intersects with $B$ transversely at finitely many points.
We assign to $D$ the orientation induced by the orientation of $L$.

Let $\alpha_B: I \to \R^3$ be an oriented path whose image is the core of $B$, with $\alpha_B(0) \in L_0$ and $\alpha_B(1) \in L_1$.
Take a partition $0 = s_0 < r_0 < s_1 < r_1 < s_2 < r_2 < \cdots < s_n < r_n < s_{n+1} = 1$ of $I$ such that for each $i = 0, 1, \dots, n$,
\begin{align*}
    \alpha_B([r_i, s_{i+1}]) \cap D_0 = \emptyset, \quad \alpha_B([s_i, r_i]) \cap D_1 = \emptyset.
\end{align*}
\begin{figure}[h]
    \centering
    \includegraphics[height = 25mm]{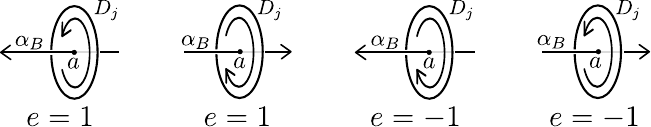}
    \caption{Local models around an intersection point $a$ and their signs $e$.}
    \label{fig:Sign-Intersection}
\end{figure}
For an intersection point $a$ of $\alpha_B(I)$ with $D$, we define the sign $e$ of $a$ as shown in \Cref{fig:Sign-Intersection}.
For $i = 1,2, \dots, n$, we define $p_i$ to be the sum of the signs of intersection points in $\alpha_B([r_{i-1}, s_i]) \cap D_0$ and $q_i$ to be the sum of the signs of intersection points in $\alpha_B([s_i, r_i]) \cap D_1$.

It is known that two ribbon 2-knots of 1-fusion with the same ribbon type are equivalent.
For any $p_i, q_i \in \Z$ ($i = 1,\dots, n$), there exists a ribbon 2-knot of 1-fusion with ribbon type $R(p_1, q_1, \dots, p_n, q_n)$.
We remark that a ribbon 2-knot can have several ribbon types \cite{Kanenobu-Takahashi2021,Yasuda2009}.

\begin{proposition}[\cite{Kinoshita1961,Marumoto1977}]\label{Formula of Alexander polynomial for ribbon 2-knots}
    Let $F$ be a ribbon 2-knot of 1-fusion with ribbon type $R(p_1, q_1, \dots, p_n, q_n)$.
    Then
    \begin{align*}
        \Delta_F(t) ~\doteq~ 1 + \sum_{i= 1}^n\left( -t^{\, a(i)} + t^{\, b(i)} \right) ~=~ 1 - t^{p_1} + t^{p_1 + q_1} - t^{p_1 + q_1 + p_2} + \cdots - t^{p_1 + q_1 + \cdots + p_n} + t^{p_1 + q_1 + \cdots + p_n + q_n},
    \end{align*}
    where $a(i) = p_i + \sum_{k = 1}^{i-1} (p_k + q_k)$, $b(i) = \sum_{k=1}^i \left( p_k + q_k\right)$, and $f(t) \doteq g(t)$ means that $f(t) = \pm t^k g(t)$ for some $k \in \Z$.
\end{proposition}

\begin{theorem}[\cite{Kinoshita1961}]\label{fact:Kinoshita1961}
    For any Laurent polynomial $f(t) \in \Z[t^{\pm1}]$ with $|f(1)| = 1$, there exists a ribbon 2-knot $F$ of 1-fusion such that $\Delta_F(t) \doteq f(t)$.
\end{theorem}

\begin{example}\label{fact:Kanenobu-Sumi2018}
    Let $F_n$ ($n \geq 0$) be a ribbon 2-knot of 1-fusion with ribbon type $R(1, n, -n-1, 1)$.
    Then, $\Delta_{F_n}(t) \doteq 1$ for any $n\geq 0$.
    Moreover, Kanenobu and Sumi \cite{Kanenobu-Sumi2018} proved that $F_n$ is equivalent to $F_m$ if and only if $n = m$.
    Therefore, there exist infinitely many ribbon 2-knots of 1-fusion with trivial Alexander polynomial.
\end{example}

The aim of this subsection is to prove the following theorem:

\begin{theorem}\label{Ribbon types of 2-plat 2-knots}
    Let $p$ and $q$ be coprime integers with $p$ odd and $q$ even.
    Then, $F(p,q)$ has a ribbon type
    \begin{align*}
        R(\varepsilon_1, \dots, \varepsilon_{p-1}), \quad \varepsilon_i = (-1)^{\lfloor iq/p \rfloor}.
    \end{align*}
\end{theorem}

\Cref{Main: Formula of Alexander polynomials for 2-plat 2-knots} follows immediately from \Cref{Formula of Alexander polynomial for ribbon 2-knots} and \Cref{Ribbon types of 2-plat 2-knots}.

\begin{proof}[Proof of \Cref{Ribbon types of 2-plat 2-knots}]
    Let $F(p,q)$ be a 2-plat 2-knot with $p$ odd and $q$ even, and let $\varepsilon_i = (-1)^{\lfloor iq/p \rfloor}$.
    We may replace $q$ by $q' = q + kp$ for some odd integer $k$, so that $q'$ is odd and satisfies $0 < |q'| < p$.
    By \Cref{2-plat 2-knots modulo p}, $F(p,q)$ is equivalent to $F(p,q')$.

    Put $\delta_i = (-1)^{\lfloor iq'/p \rfloor}$.
    Then, we have
    \begin{align*}
        \delta_i = (-1)^{\lfloor iq'/p \rfloor} ~=~ (-1)^{\lfloor ik + (iq/p) \rfloor} ~=~ (-1)^i(-1)^{\lfloor iq/p \rfloor} ~=~ (-1)^i \varepsilon_i.
    \end{align*}

    In the following, we use some notation introduced in the proof of \Cref{Relationship between Schubert and Conway normal forms}.
    Let $[c_1, \dots, c_m]$ be the canonical continued fraction expansion of $q'/p$, and let $\beta_c$ be the geometric 4-braid in $\R^2 \times [0,1]$.
    Let $\beta_c^{-1}$ be the mirror image of $\beta_c$ about the plane $\R^3 \times \{0\}$, which is a geometric 4-braid in $\R^2 \times [-1,0]$.
    We define $(L, B)$ as a ribbon presentation shown in the left-hand side of \Cref{fig:RP-2plat2knot}, where $L$ is the oriented plat closure of $\beta_c \cup \beta_c^{-1}$ and $B$ is a single band whose core lies in the $xy$-plane.
    By \Cref{Ribbon presentation of plat closure}, $(L, B)$ is a ribbon presentation of $F(p,q')$.

    Let $\{ \varPhi_t\}_{t \in I}$ be the isotopy of $\R^3$ defined in the proof of \Cref{Relationship between Schubert and Conway normal forms}.
    Define an isotopy $\varPsi_t: \R^3 \to \R^3$ ($t \in I$) by $\varPsi_t = \varPhi_{1-t} \circ \varPhi^{-1}_1$.
    Then, $(L', B') = (\varPsi_1(L), \varPsi_1(B))$ is a new ribbon presentation of $F(p,q')$ with the following properties:
    \begin{itemize}
        \item $L' = L_0 \cup L_1$ is embedded in the yz-plane $\{0\} \times \R^2$,
        \item the orientation of each component of $L'$ is counter-clockwise, and
        \item the core of $B'$ is an under-path of $K(p,q')$.
    \end{itemize}
    See \Cref{fig:RP-2plat2knot}, where the middle of \Cref{fig:RP-2plat2knot} shows the case of $(p,q') = (3,1)$, and the right-hand side of \Cref{fig:RP-2plat2knot} shows that of $(p,q') = (7,3)$.

    \begin{figure}[h]
        \centering
        \begin{minipage}{0.3\hsize}
            \centering
            \includegraphics[height = 40mm]{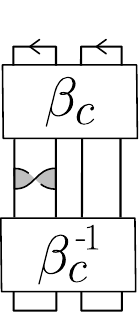}
        \end{minipage}
        \begin{minipage}{0.3\hsize}
            \centering
            \includegraphics[height = 40mm]{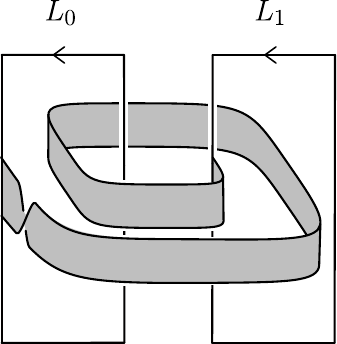}
        \end{minipage}
        \begin{minipage}{0.3\hsize}
            \centering
            \includegraphics[height = 40mm]{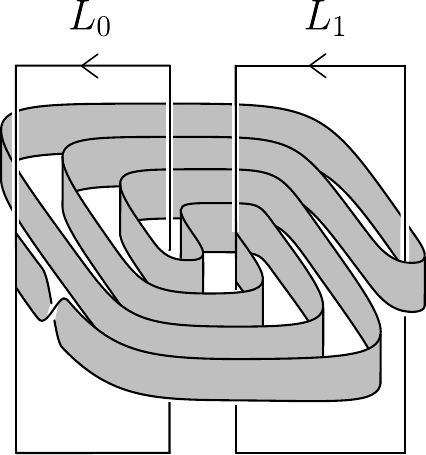}
        \end{minipage}
        \caption{\textbf{Left}: A ribbon presentation $(L, B)$ of $F(p,q')$, \textbf{Middle}: A ribbon presentation $(L', B')$ of $F(3,1)$, \textbf{Right}: A ribbon presentation $(L', B')$ of $F(7,3)$.}
        \label{fig:RP-2plat2knot}
    \end{figure}

    Let $D = D_0 \cup D_1$ be the union of 2-disks in the yz-plane with $\partial D_i = L_i$.
    Let $\alpha_{B'}: I \to \R^3$ be an oriented path such that $\alpha_{B'}(0) \in L_0$, $\alpha_{B'}(1) \in L_1$, and $\alpha_{B'}(I)$ is the core of $B'$.
    Then, the intersection points $\alpha_{B'}(I) \cap D$ correspond to the crossings in the diagram $D(p,q')$ of $K(p,q')$.
    Let $e_i$ be the sign of the $i$-th intersection point on $\alpha_{B'}(I)$.
    Then, we have
    \begin{align*}
        e_i ~=~ (-1)^i \delta_i ~=~ \varepsilon_i.
    \end{align*}
    Since $\alpha_{B'}(I)$ passes through $D_0$ and $D_1$ alternately, $F(p,q')$ has a ribbon type $R(\varepsilon_1, \dots, \varepsilon_{p-1})$.
\end{proof}

\subsection{The invariant $\tau(F)$ of surface-links}\label{Section:tau invariant}

In this subsection, we introduce two invariants $a(F)$ and $\tau(F)$ of oriented surface-links.
\begin{definition}
    Let $F$ be an oriented surface-link, and let $\Delta_F(t) = \sum_{k \in \Z} a_kt^k$ be the Alexander polynomial of $F$.
    We define $a(F)$ and $\tau(F)$ as
    \begin{align*}
        a(F) ~=~ \sum_{k \in \Z} |a_k|, \quad \tau(F) ~=~ \sum_{k \in \Z} k \, |a_k| \pmod{a(F)}.
    \end{align*}
\end{definition}

\begin{lemma}\label{well-definedness of tau}
    $a(F)$ and $\tau(F)$ are well-defined invariants of oriented surface-links.
\end{lemma}

\begin{proof}
    Let $F$ and $F'$ be equivalent surface-links.
    Then, we have $\Delta_F(t) \doteq \Delta_{F'}(t)$, that is, $\Delta_{F'}(t) = \varepsilon t^{n}\Delta_{F}(t)$ for some $n \in \Z$ and $\varepsilon \in \{\pm 1\}$.
    In other words, $\Delta_{F'}(t) =  \sum_{k \in \Z} \varepsilon a_kt^{k+n}$.

    It is immediate from the definition that $a(F)$ and $a(F')$ coincide.
    Thus, we have
    \begin{align*}
        \tau(F') ~&=~ \sum_{k \in \Z} (k+n) \, |\varepsilon a_k| ~=~
        \sum_{k \in \Z} k \, |a_k| + \sum_{k \in \Z} n \, |a_k| ~=~
        \sum_{k \in \Z} k \, |a_k| + n \sum_{k \in \Z} |a_k|\\
        &=~ \tau(F) + na(F) ~\equiv~ \tau(F) \pmod{a(F)}.
    \end{align*}
\end{proof}

We now prove \Cref{Main:Formula of tau invariant}.

\begin{proof}[Proof of \Cref{Main:Formula of tau invariant}]
    Let $F=F(p,q)$ be an oriented 2-plat 2-knot, where $p > 0$ is odd and $q$ is even, and let $\Delta_F(t) = \sum_{k \in \Z} a_kt^k$ be the Alexander polynomial of $F$.
    If $q = 0$, then we have $p = 1$ and $\tau(F(1,0)) = 0$ by definition.
    Thus, we only consider the case $|q| > 1$.

    The proof consists of two steps:
    First, we prove $a(F) = p$. Second, we rewrite $\tau(F)$ in terms of the signs $\varepsilon_i$ and evaluate the resulting sum modulo $p$.

    \begin{claim}\label{claim1}
        We have $a(F) = p$.
    \end{claim}

    \begin{proof}[Proof of \Cref{claim1}]
        By \Cref{Main: Formula of Alexander polynomials for 2-plat 2-knots}, $\Delta_F(t)$ is
        \begin{align*}
            \Delta_F(t) ~\doteq~ 1 + \sum_{i = 1}^{p-1} (-1)^i t^{\, d(i)}, \quad d(i) ~=~ \sum_{l=1}^{i} \varepsilon_l ~=~ \sum_{l=1}^i (-1)^{\lfloor lq/p\rfloor}.
        \end{align*}
        Since $d(i) \equiv i \pmod 2$, we have $a_k = (-1)^k|a_k|$.
        Thus, we have
        \begin{align*}
            \Delta_F(t) ~\doteq~ 1 + \sum_{i = 1}^{p-1} (-1)^{i} t^{\,d(i)} ~=~ \sum_{k \in \Z} |a_k| (-1)^k t^{\, k}.
        \end{align*}
        By \Cref{Determinant of 2-plat 2-knot}, we have
        \begin{align*}
            p ~=~ |\Delta_F(-1)| ~=~ \left| \sum_{k \in \Z} |a_k| (-1)^k (-1)^k \right| ~=~ \sum_{k \in \Z} |a_k| ~=~ a(F).
        \end{align*}
    \end{proof}

    \begin{claim}\label{claim2}
        We have $\tau(F(p,q)) ~\equiv~ - \sum_{i=1}^{p-1} i\, \varepsilon_i \pmod p$.
    \end{claim}

    \begin{proof}[Proof of \Cref{claim2}]
        Let $N_k = \left\{ i \mid d(i) = k \right\} \subset \{1,2, \dots, p-1\}$, where $k \in \Z$.
        Since $d(i) \equiv i \pmod 2$, we have $|a_k| = |N_k|$ for each $k$.
        In addition, $N_k \cap N_{k'} = \emptyset$ for $k \neq k'$.
        Thus, we have
        \begin{align*}
            \tau(F) ~&=~
            \sum_{k \in \Z} k \, |a_k| ~=~
            \sum_{k \in \Z} \sum_{i \in N_k} k ~=~
            \sum_{k \in \Z} \sum_{i \in N_k} d(i) ~=~
            \sum_{i=1}^{p-1} d(i)\\
            ~&=~
            \sum_{i=1}^{p-1}\sum_{l=1}^{i} \varepsilon_l ~=~
            \sum_{i=1}^{p-1} (p-i) \varepsilon_i ~\equiv~
            -\sum_{i=1}^{p-1} i\, \varepsilon_i \pmod p.
        \end{align*}
    \end{proof}

    We continue the proof of \Cref{Main:Formula of tau invariant}.
    Write $q=2q_0$ and $iq_0 \equiv r_i \pmod p$ with $0\le r_i\le p-1$.
    Then, we have
    \[
        \varepsilon_i ~=~ (-1)^{\lfloor iq/p\rfloor}~=~ (-1)^{\lfloor 2r_i/p \rfloor}
        ~=~
        \begin{cases}
            +1 & 1 \leq r_i \leq (p-1)/2,\\
            -1 & (p+1)/2 \leq r_i \leq p-1.
        \end{cases}
    \]
    Define $\chi: \{ 1,2,\dots,p-1 \} \to \{ \pm1 \}$ by
    \[
        \chi(r) ~=~
        \begin{cases}
            +1 & 1\le r\le (p-1)/2,\\
            -1 & (p+1)/2\le r\le p-1.
        \end{cases}
    \]
    Then, we have $\varepsilon_i = \chi(r_i)$.
    In addition, we have $\{r_i \mid i = 1,\dots, p-1\} = \{1, 2, \dots, p-1\}$ since $\gcd(p, q_0) = 1$.
    Hence
    \begin{align*}
        -q_0 \tau(F) ~&=~
        q_0 \sum_{i=1}^{p-1} i\,\varepsilon_i ~=~
        \sum_{i=1}^{p-1} i q_0 \,\chi(r_i) ~\equiv~
        \sum_{i=1}^{p-1} r_i\,\chi(r_i) \pmod p\\
        ~&=~
        \sum_{r=1}^{p-1} r\,\chi(r) ~=~
        \sum_{r=1}^{(p-1)/2} r - \sum_{r=(p+1)/2}^{p-1} r ~=~
        - \frac{(p-1)^2}{4} ~\equiv~
        -4^{-1} \pmod p.
    \end{align*}
    Therefore, we have $\tau(F) \equiv (4q_0)^{-1} = (2q)^{-1} \pmod p$.
\end{proof}

\appendix
\section{A table of 2-plat 2-knots with small determinant}

In \Cref{table: 2-plat 2-knots}, we list the 2-plat 2-knots $F(p,q)$ with $p \leq 19$.
By \Cref{2-plat 2-knots modulo p}, every non-trivial 2-plat 2-knot has a normal form $F(p,q)$ with $0 < |q| < p$ and $q$ even.
In addition, we remove $F(p,q)$ with $q < 0$ from \Cref{table: 2-plat 2-knots} because $F(p,q)$ is the mirror image of $F(p,-q)$.
Each column in \Cref{table: 2-plat 2-knots} is as follows:
\begin{itemize}
    \item The column, $q/p$, shows the rational number $q/p$ of a normal form $F(p,q)$.
    \item The column, Ribbon type, shows the ribbon type of $F(p,q)$ as determined in \Cref{Ribbon types of 2-plat 2-knots}.
    \item The column, $\Delta(t)$, shows the normalized Alexander polynomial of $F(p,q)$.
    Here, the Alexander polynomial $\Delta(t)$ is said to be \textit{normalized} if it satisfies $\Delta(1) = 1$ and $(d/dt)\Delta(1) = 0$.
    We abbreviate $\Delta(t)$ as follows: $(a_m, a_{m+1}, \dots, [a_0], a_1, \dots, a_n) = \sum_{i=m}^n a_i t^{\, i}$.
    \item In the last column, the symbol $\approx K$ indicates that $F(p,q)$ is equivalent to a ribbon 2-knot $K$ classified by Kanenobu and Takahashi \cite{Kanenobu-Takahashi2021}.
\end{itemize}

\bibliographystyle{plain}
\bibliography{reference.bib}

\begin{table}[h]
\begin{tabular}{rlcc}\hline
$a/p$ & Ribbon type & $\Delta(t)$\\ \hline
0/1  &  $R (0, 0) $ (Trivial 2-knot) & $ ([1]) $ & $\approx 0\_1$ \\
\\
2/3  &  $R (1, -1) $ & $ ([0], 2, -1) $ & $\approx 2\_2$ \\
\\
2/5  &  $R (1, 1, -1, -1) $ & $ ([2], -2, 1) $ & $\approx 4\_8$ \\
4/5  &  $R (1, -1, 1, -1) $ & $ ([0], 0, 3, -2) $ & $\approx 4\_9$ \\
\\
2/7  &  $R (1, 1, 1, -1, -1, -1) $ & $ ([0], 2, -2, 2, -1) $ & $\approx 6\_64$ \\
4/7  &  $R (1, -1, -1, 1, 1, -1) $ & $ ([-1], 4, -2) $ & $\approx 6\_69$ \\
6/7  &  $R (1, -1, 1, -1, 1, -1) $ & $ ([0], 0, 0, 4, -3) $ & $\approx 6\_68$ \\
\\
2/9  &  $R (1, 1, 1, 1, -1, -1, -1, -1) $ & $ ([2], -2, 2, -2, 1) $ &  \\
4/9  &  $R (1, 1, -1, -1, 1, 1, -1, -1) $ & $ ([3], -4, 2) $ &  \\
8/9  &  $R (1, -1, 1, -1, 1, -1, 1, -1) $ & $ ([0], 0, 0, 0, 5, -4) $ &  \\
\\
2/11  &  $R (1, 1, 1, 1, 1, -1, -1, -1, -1, -1) $ & $ ([0], 2, -2, 2, -2, 2, -1) $ &  \\
4/11  &  $R (1, 1, -1, -1, -1, 1, 1, 1, -1, -1) $ & $ (-1, 4, [-4], 2) $ &  \\
6/11  &  $R (1, -1, -1, 1, 1, -1, -1, 1, 1, -1) $ & $ ([-2], 6, -3) $ &  \\
8/11  &  $R (1, -1, 1, 1, -1, 1, -1, -1, 1, -1) $ & $ ([0], 4, -5, 2) $ &  \\
10/11  &  $R (1, -1, 1, -1, 1, -1, 1, -1, 1, -1) $ & $ ([0], 0, 0, 0, 0, 6, -5) $ &  \\
\\
2/13  &  $R (1, 1, 1, 1, 1, 1, -1, -1, -1, -1, -1, -1) $ & $ ([2], -2, 2, -2, 2, -2, 1) $ &  \\
4/13  &  $R (1, 1, 1, -1, -1, -1, 1, 1, 1, -1, -1, -1) $ & $ ([0], 0, 3, -4, 4, -2) $ &  \\
6/13  &  $R (1, 1, -1, -1, 1, 1, -1, -1, 1, 1, -1, -1) $ & $ ([4], -6, 3) $ &  \\
8/13  &  $R (1, -1, -1, 1, -1, -1, 1, 1, -1, 1, 1, -1) $ & $ (1, -4, [6], -2) $ &  \\
10/13  &  $R (1, -1, 1, -1, -1, 1, -1, 1, 1, -1, 1, -1) $ & $ ([0], -2, 7, -4) $ &  \\
12/13  &  $R (1, -1, 1, -1, 1, -1, 1, -1, 1, -1, 1, -1) $ & $ ([0], 0, 0, 0, 0, 0, 7, -6) $ &  \\
\\
2/15  &  $R (1, 1, 1, 1, 1, 1, 1, -1, -1, -1, -1, -1, -1, -1) $ & $ ([0], 2, -2, 2, -2, 2, -2, 2, -1) $ &  \\
4/15  &  $R (1, 1, 1, -1, -1, -1, -1, 1, 1, 1, 1, -1, -1, -1) $ & $ ([-1], 4, -4, 4, -2) $ &  \\
8/15  &  $R (1, -1, -1, 1, 1, -1, -1, 1, 1, -1, -1, 1, 1, -1) $ & $ ([-3], 8, -4) $ &  \\
14/15  &  $R (1, -1, 1, -1, 1, -1, 1, -1, 1, -1, 1, -1, 1, -1) $ & $ ([0], 0, 0, 0, 0, 0, 0, 8, -7) $ &  \\
\\
2/17  &  $R (1, 1, 1, 1, 1, 1, 1, 1, -1, -1, -1, -1, -1, -1, -1, -1) $ & $ ([2], -2, 2, -2, 2, -2, 2, -2, 1) $ &  \\
4/17  &  $R (1, 1, 1, 1, -1, -1, -1, -1, 1, 1, 1, 1, -1, -1, -1, -1) $ & $ ([3], -4, 4, -4, 2) $ &  \\
6/17  &  $R (1, 1, -1, -1, -1, 1, 1, 1, -1, -1, -1, 1, 1, 1, -1, -1) $ & $ (-2, 6, -6, [3]) $ &  \\
8/17  &  $R (1, 1, -1, -1, 1, 1, -1, -1, 1, 1, -1, -1, 1, 1, -1, -1) $ & $ ([5], -8, 4) $ &  \\
10/17  &  $R (1, -1, -1, 1, 1, -1, 1, 1, -1, -1, 1, -1, -1, 1, 1, -1) $ & $ ([0], -2, 8, -6, 1) $ &  \\
12/17  &  $R (1, -1, 1, 1, -1, 1, 1, -1, 1, -1, -1, 1, -1, -1, 1, -1) $ & $ ([0], 0, 4, -6, 5, -2) $ &  \\
14/17  &  $R (1, -1, 1, -1, 1, 1, -1, 1, -1, 1, -1, -1, 1, -1, 1, -1) $ & $ ([0], 0, 6, -8, 3) $ &  \\
16/17  &  $R (1, -1, 1, -1, 1, -1, 1, -1, 1, -1, 1, -1, 1, -1, 1, -1) $ & $ ([0], 0, 0, 0, 0, 0, 0, 0, 9, -8) $ &  \\
\\
2/19  &  $R (1, 1, 1, 1, 1, 1, 1, 1, 1, -1, -1, -1, -1, -1, -1, -1, -1, -1) $ & $ ([0], 2, -2, 2, -2, 2, -2, 2, -2, 2, -1) $ &  \\
4/19  &  $R (1, 1, 1, 1, -1, -1, -1, -1, -1, 1, 1, 1, 1, 1, -1, -1, -1, -1) $ & $ (-1, 4, [-4], 4, -4, 2) $ &  \\
6/19  &  $R (1, 1, 1, -1, -1, -1, 1, 1, 1, -1, -1, -1, 1, 1, 1, -1, -1, -1) $ & $ ([0], 0, 0, 4, -6, 6, -3) $ &  \\
8/19  &  $R (1, 1, -1, -1, 1, 1, 1, -1, -1, 1, 1, -1, -1, -1, 1, 1, -1, -1) $ & $ ([0], 4, -7, 6, -2) $ &  \\
10/19  &  $R (1, -1, -1, 1, 1, -1, -1, 1, 1, -1, -1, 1, 1, -1, -1, 1, 1, -1) $ & $ ([-4], 10, -5) $ &  \\
12/19  &  $R (1, -1, -1, 1, -1, -1, 1, -1, -1, 1, 1, -1, 1, 1, -1, 1, 1, -1) $ & $ (-1, 4, [-6], 6, -2) $ &  \\
14/19  &  $R (1, -1, 1, 1, -1, 1, -1, -1, 1, -1, 1, 1, -1, 1, -1, -1, 1, -1) $ & $ ([0], 6, -9, 4) $ &  \\
16/19  &  $R (1, -1, 1, -1, 1, -1, -1, 1, -1, 1, -1, 1, 1, -1, 1, -1, 1, -1) $ & $ ([0], 0, -3, 10, -6) $ &  \\
18/19  &  $R (1, -1, 1, -1, 1, -1, 1, -1, 1, -1, 1, -1, 1, -1, 1, -1, 1, -1) $ & $ ([0], 0, 0, 0, 0, 0, 0, 0, 0, 10, -9) $ &  \\
\hline
\end{tabular}
\centering
\caption{2-plat 2-knots $F(a/p)$ with $p \leq 19$.}
\label{table: 2-plat 2-knots}
\end{table}
\end{document}